\documentclass[reqno]{amsart}
\usepackage{amssymb,amsmath,hyperref}
\usepackage{amsrefs}
\usepackage[foot]{amsaddr}
\usepackage{bbold,stackrel}
 
\newcommand{\Z}{{\textsf{\textup{Z}}}}

\newcommand{\dminus}{\mbox{$\;^\cdot\!\!\!-$}}

\newtheorem{thm}{Theorem}
\newtheorem{lem}[thm]{Lemma}
\newtheorem{cor}[thm]{Corollary}
\newtheorem{defi}[thm]{Definition}
\newtheorem{rem}[thm]{Remark}
\newtheorem{nota}[thm]{Notation}
\newtheorem{exa}[thm]{Example}

\newtheorem{ack}[thm]{Acknowledgement}

\newtheorem*{tempo*}{Template}
\newtheorem{theorem}[thm]{Theorem}

\newtheorem{lemma}[thm]{Lemma}

\newtheorem{convention}[thm]{Convention}

\newcommand\be{\begin{equation}}
\newcommand\ee{\end{equation}} 

\usepackage{amsmath,amsfonts} 

\usepackage[applemac]{inputenc}

\def\bdefi{\begin{defi}\rm}
\def\edefi{\end{defi}}
\def\bnota{\begin{nota}\rm}
\def\enota{\end{nota}}
\def\FIVE{\Pi_{1}^{1}\text{-\textup{\textsf{CA}}}_{0}}
\def\SIX{\Pi_{2}^{1}\text{-\textsf{\textup{CA}}}_{0}}
\def\SIXk{\Pi_{k}^{1}\text{-\textsf{\textup{CA}}}_{0}}
\def\SIXK{\Pi_{k}^{1}\text{-\textsf{\textup{CA}}}_{0}^{\omega}}

\def\ATR{\textup{\textsf{ATR}}}
\def\ZFC{\textup{\textsf{ZFC}}}
\def\ZF{\textup{\textsf{ZF}}}

\def\L{\textsf{\textup{L}}}

\def\RCA{\textup{\textsf{RCA}}}
\def\({\textup{(}}
\def\){\textup{)}}

\def\RCAo{\textup{\textsf{RCA}}_{0}^{\omega}}
\def\ACAo{\textup{\textsf{ACA}}_{0}^{\omega}}
\def\WKL{\textup{\textsf{WKL}}}

\def\bye{\end{document}}
\def\N{{\mathbb  N}}
\def\Q{{\mathbb  Q}}
\def\R{{\mathbb  R}}

\def\di{\rightarrow}

\def\asa{\leftrightarrow}

\def\ACA{\textup{\textsf{ACA}}}

\def\QFAC{\textup{\textsf{QF-AC}}}

\def\NFP{\textup{\textsf{NFP}}}

\def\HBU{\textup{\textsf{HBU}}}

\def\net{\textup{\textsf{net}}}
\def\BW{\textup{\textsf{BW}}}
\def\FTC{\textup{\textsf{FTC}}}
\def\seq{\textup{\textsf{seq}}}

\def\LIND{\textup{\textsf{LIND}}}
\def\LIN{\textup{\textsf{LIN}}}

\def\SJ{\mathbb{S}}

\def\eps{\varepsilon}

\def\ECF{\textup{\textsf{ECF}}}

\def\SCF{\textup{\textsf{SCF}}}

\usepackage{graphicx}
\usepackage{tikz}
\usetikzlibrary{matrix, shapes.misc}

\numberwithin{equation}{section}
\numberwithin{thm}{section}

\usepackage{comment}

\begin{document}
\title[On the significance of the uncountable]{On the mathematical and foundational significance of the uncountable}
\author{Dag Normann}
\address{Department of Mathematics, The University 
of Oslo, P.O. Box 1053, Blindern N-0316}
\email{dnormann@math.uio.no}
\author{Sam Sanders}
\address{School of Mathematics, University of Leeds \& Dept.\ of Mathematics, TU Darmstadt}
%\address{Center for Advanced Studies and MCMP, LMU Munich, Germany}
\email{sasander@me.com}
%\subjclass[2010]{03B30, 03D65, 03F35}
%\keywords{reverse mathematics, higher-order computability theory, higher-order arithmetic, Gandy's superjump, gauge integral}
\begin{abstract}
We study the logical and computational properties of basic theorems of uncountable mathematics, including the \emph{Cousin} and \emph{Lindel\"of lemma} published in 1895 and 1903.  Historically, these lemmas were among the first formulations of \emph{open-cover compactness} and the \emph{Lindel\"of property}, respectively.
These notions are of great conceptual importance: the former is commonly viewed as a way of treating uncountable sets like e.g.\ $[0,1]$ as `almost finite', while the latter allows one to treat uncountable sets like e.g.\ $\R$ as `almost countable'. 
This reduction of the uncountable to the finite/countable turns out to have a \emph{considerable} logical and computational cost: we show that the aforementioned lemmas, and many related theorems, are \emph{extremely} hard to prove, while the associated sub-covers are \emph{extremely} hard to compute.
Indeed, in terms of the standard scale (based on comprehension axioms), a proof of these lemmas requires at least the full extent of \emph{second-order arithmetic}, a system originating from Hilbert-Bernays' \emph{Grundlagen der Mathematik}.  This observation has far-reaching implications for the \emph{Grundlagen}'s spiritual successor, the program of \emph{Reverse Mathematics}, and the associated \emph{G\"odel hierachy}. 
We also show that the Cousin lemma is essential for the development of the \emph{gauge integral}, a generalisation of the Lebesgue and improper Riemann integrals that also uniquely provides a direct formalisation of Feynman's path integral.  
\end{abstract}

%\setcounter{page}{0}
%\tableofcontents
%\thispagestyle{empty}
%\newpage

\maketitle
\thispagestyle{empty}
\vspace{-0.5cm}
%From time immemorial, the infinite has stirred men's emotions more than any other question. Hardly any other idea has stimulated the mind so fruitfully. Yet, no other concept needs clarification more than it does. (Hilbert, On the infinite)
The content of this paper extends \cite{dagsamIII} in that Sections \ref{extradata} and \ref{klipel} below are new.
Small corrections/additions have also been made to reflect new developments.
\section{Introduction}
\subsection{Infinity: hubris and catharsis}
It is a commonplace that finite and countable sets exhibit many useful properties that uncountable sets lack.  
Conveniently, there are properties that allow one to treat uncountable sets \emph{as though they were finite or countable}, namely \emph{open-cover compactness} and the \emph{Lindel\"of property}, i.e.\  the statement that an open cover has a finite, respectively countable, sub-cover.  

\smallskip

These notions are well-established: the \emph{Cousin lemma} (\cite{cousin1}*{p.\ 22}) on the open-cover compactness of subsets of $\R^{2}$, 
dates back\footnote{The collected works of Pincherle contain a footnote by the editors (See \cite{tepelpinch}*{p.\ 67}) which states that the associated \emph{Teorema} (published in 1882) corresponds to the Heine-Borel theorem.  Moreover, Weierstrass proves the Heine-Borel theorem (without explicitly formulating it) in 1880 in \cite{amaimennewekker}*{p.\ 204}.   A detailed motivation for these claims may be found in \cite{medvet}*{p. 96-97}.} 135 years, while the \emph{Lindel\"of lemma} (\cite{blindeloef}*{p.\ 698}) on the \emph{Lindel\"of property} of $\R^{n}$, dates back about 115 years.  
Despite their basic nature, their central role in analysis, and a long history, little is known about the logical and computational properties of the Cousin and Lindel\"of lemmas.  In a nutshell, we aim to fill this hole in the literature in this paper.  We discuss our motivations and goals in detail in Sections \ref{foma} and \ref{basic} respectively.   

\smallskip

As it tuns out, the \emph{hubris} of reducing the uncountable to the finite/countable as in the Cousin and Lindel\"of lemmas, comes at great logical and computational cost.  
Indeed, we establish below that these lemmas are \emph{extremely hard to prove}, while the sub-covers from these lemmas are similarly \emph{hard to compute}. 
Now, `hardness of proof' is measured by what \emph{comprehension\footnote{Intuitively speaking, a comprehension axiom states that the set $\{x\in X: \varphi(x)\}$ exists for all formulas $\varphi$ in a certain class, and with the variable $x$ in the domain $X$.} axioms} are necessary to prove the theorem.  
In this sense, a proof of the Cousin and Lindel\"of lemmas requires (comprehension axioms as strong as) \emph{second-order arithmetic}, as is clear from Figure~\ref{xxy}, where
the latter originates from Hilbert-Bernays' \emph{Grundlagen der Mathematik} \cite{hillebilly2}.

\smallskip

Moreover, the Cousin and Lindel\"of lemmas are not isolated events: we provide a list of basic theorems (See Section~\ref{basic}) with the same `extreme' logical and computational properties.
Some of the listed theorems are even of great conceptual importance as they pertain to the \emph{gauge integral} (\cite{bartle}), which provides a generalisation of the Lebesgue and improper Riemann integrals, and (to the best of our knowledge) the only \emph{direct} formalisation of the \emph{Feynman path integral} (\cites{burkdegardener,mullingitover, mully}).     

\smallskip

By way of \emph{catharsis}, our results call into question various empirical claims from the foundation of mathematics, such as the `Big Five' classification from Reverse Mathematics (See Section \ref{RM}) and the linear nature of the \emph{G\"odel hierarchy} (See Section \ref{kodel}).  
Nonetheless, we obtain in Section \ref{introgau} \emph{Reverse Mathematics style} equivalences
involving the Cousin lemma and \emph{basic} properties of the \emph{gauge integral}.  % like its uniqueness and its extension of the Riemann integral.  

\smallskip

Finally, Reverse Mathematics is intimately connected to \emph{classical} computability theory (See e.g.\ \cite{simpson2}*{II.7.5}); similarly, our results have an (almost) equivalent reformulation in \emph{higher-order} computability theory, and are even (often) obtained via the latter.    
Furthermore, in light of this correspondence, we investigate in Section~\ref{RMR2} the strength of the Cousin and Lindel\"of lemmas \emph{when combined with} fundamental objects from computability theory.  This study yields surprising results reaching all the way up to Gandy's \emph{superjump} (\cite{supergandy}), a `higher-order' version of Turing's \emph{Halting problem} (\cite{tur37}), the prototypical non-computable object.

\subsection{Foundational and mathematical motivations}\label{foma}
We discuss the motivations for this paper.  Items \eqref{forko} and \eqref{forko2} motivate the study of mathematics \emph{beyond} the language of second-order arithmetic $\L_{2}$, the framework for `classical' Reverse Mathematics, while a notable consequence is provided by item \eqref{forko3}.  
\begin{enumerate}
\renewcommand{\theenumi}{\roman{enumi}}
\item The \emph{gauge integral} is a generalisation of the Lebesgue and (improper) Riemann integral, and formalises Feynman's path integral (See Section~\ref{mamo}).  
The language $\L_{2}$ cannot accommodate (basic) gauge integration.\label{forko}
\item The foundational studies of mathematics led by Hilbert take place in a logical framework \emph{richer} than the language $\L_{2}$ (See Section \ref{fomo}).  
It is natural to ask if anything is lost by restricting to $\L_{2}$.  \label{forko2}
\item The \emph{compatibility problem} for Nelson's \emph{predicative arithmetic} (\cite{ohnelly}) was solved in the negative (\cite{buss3}).  We solve the compatibility problem for \emph{Weyl-Feferman predicative mathematics} in the negative (See Section \ref{comp}).  \label{forko3}
\end{enumerate}
As an example of how items \eqref{forko} and \eqref{forko2} are intimately related: the \emph{uniqueness} of the gauge integral requires (Heine-Borel) compactness for \emph{uncountable} covers.  
The latter compactness cannot be formulated in $\L_{2}$, and will be seen to have \emph{completely} different logical and computational properties compared to the `countable/second-order' substitute, i.e.\ (Heine-Borel) compactness for \emph{countable} covers.  

\

\subsubsection{Mathematical motivations}\label{mamo}
In this section, we discuss the mathematical motivations for this paper, provided by the study of the \emph{gauge integral}.  
As will become clear,  the latter cannot be (directly) formulated in the language of second-order arithmetic, yielding a measure of motivation 
for our adoption of Kohlenbach's \emph{higher-order} framework involving all finite types.  

\smallskip

First of all, the gauge integral (aka \emph{Henstock-Kurzweil} integral) was introduced around 1912 by Denjoy (in a different form) and constitutes a simultaneous generalisation of the Lebesgue and \emph{improper} Riemann integral.  
The gauge integral provides (to the best of our knowledge) the only formal framework close to the original development of the \emph{Feynman path integral} (\cite{mullingitover, burkdegardener, mully}), i.e.\ gauge integrals are highly relevant in (the foundations of) physics.   
As expected, the gauge integral can handle discontinuous functions, which were around at the time: Dirichlet discusses the characteristic function of $\Q$ around 1829 in \cite{didi1}, while Riemann defines a function with countably many discontinuities in his \emph{Habilitationsschrift} 
 \cite{kleine}.  

\smallskip

Secondly, since Lebesgue integration is studied in Reverse Mathematics (See \cite{simpson2}*{X.1}), it is a natural next step to study the gauge integral.  
However, this study cannot take place in the language of second-order arithmetic for the following reasons: on one hand (general) discontinuous functions are essential for proving basic results of the gauge integral by Remark \ref{splitskop} and Corollary \ref{ofvaluesee}.  On the other hand, by Theorem~\ref{firstje}, the \emph{uniqueness} of the gauge integral requires the \emph{Cousin lemma} (\cite{cousin1}*{p.\ 22}), which deals with \emph{uncountable covers}, and the latter cannot be formulated in the language of second-order arithmetic.  

\smallskip

In conclusion, the gauge integral seems to require a logical framework \emph{richer} than second-order arithmetic.  
Now, this richer framework yields surprising results: the Cousin lemma expresses compactness for \emph{uncountable} open covers; this lemma turns out to have \emph{completely} different logical and computational properties compared to compactness restricted to \emph{countable} covers as in Reverse Mathematics (\cite{simpson2}*{IV.1}).

\subsubsection{Foundational motivations}\label{fomo}
We show that the foundational studies of mathematics led by Hilbert took place in a framework \emph{richer} than second-order arithmetic.
First of all, in his 1917-1933 lectures on the foundations of mathematics (\cite{hillebillen}), Hilbert used a logical system involving \emph{third-order}\footnote{In the notation of this paper, to be introduced in Section \ref{KOH}, $n+1$-th order objects ($n\geq 0$) correspond to objects of type $n$.\label{typisch}} \emph{Funktionfunktionen}.  Ackermann's 1924 dissertation (supervised by Hilbert) starts with an overview of \emph{Hilbertsche Beweistheorie}, i.e.\ Hilbertian proof theory, which explicitly includes \emph{third-order}$^{\ref{typisch}}$ parameters and the `epsilon' operator.  

\smallskip

Secondly, Hilbert and Bernays introduce\footnote{All other systems in \cite{hillebilly2}*{Suppl.~IV} are either a variation of $H$ or more limited than $H$.} the formal system $H$ in \cite{hillebilly2}*{Supplement~IV}, and use it to formalise parts of mathematics, again based on the `epsilon' operator.  
Now, Hilbert and Bernays in \cite{hillebilly2}*{p.\ 495} use the epsilon operator to define a certain object $\xi$ which maps functions to functions, i.e.\ a \emph{third-order object}.  % not available in second-order arithmetic in particular.  
Similarly, Feferman's `$\mu$' operator (See Section \ref{HCT}) is defined with the same name in \cite{hillebilly2}*{p.\ 476}, while the `$\nu$' operator from \cite{hillebilly2}*{p.\ 479} is only a slight variation of the Suslin functional (See Section \ref{HCT}).  
Hence, one could develop large parts of Kohlenbach's \emph{higher-order Reverse Mathematics} (See Section \ref{KOH}) in $H$.

\smallskip

Thirdly, Simpson positions \emph{Reverse Mathematics} (See Section \ref{RM}) in \cite{simpson2}*{p.\ 6} as a continuation of Hilbert-Bernays' research, namely as follows:
\begin{quote}
The development of a portion of ordinary
mathematics within [second-order arithmetic] $\textsf{Z}_{2}$ is outlined in Supplement IV of Hilbert/Bernays [\dots]. The present book may be regarded as a continuation of the research
begun by Hilbert and Bernays. 
\end{quote}
In conclusion, the foundational studies of Hilbert-Bernays-Ackermann take place in a language \emph{richer} than $\L_{2}$, 
and it is a natural \emph{foundational} question if anything is lost by restricting to the latter.  By Theorem \ref{main1}, the loss can be extreme: in terms of comprehension axioms, a proof of the Cousin lemma requires a system as strong as second-order arithmetic, 
while this lemma \emph{restricted to countable covers/the language of second-order arithmetic} is provable in a weak system by \cite{simpson2}*{IV.1}. 

\subsubsection{Foundational consequences}\label{comp}
We discuss the \emph{compatibility problem} for predicative mathematics \emph{\`a la} Weyl-Fefermann.  
As it turns out, our results solve this problem in the negative, providing another motivation for 
this paper.  

\smallskip

Russell famously identified an inconsistency in early set theory, known as \emph{Russel's paradox}, based on the `set of all sets' (\cite{vajuju}).  
According to Russel, the source of this paradox was \emph{circular reasoning}: in defining the `set of all sets', one quantifies over all sets, including
the one that is being defined.  To avoid such problems, Russel suggested banning any \emph{impredicative definition}, i.e.\ a definition in which one
quantifies over the object being defined.  The textbook example of an impredicative definition is the supremum of a bounded set of reals, defined as the \emph{least} upper bound of that set.  
Weyl, a student of Hilbert, initiated the development of \emph{predicative mathematics} (\cite{weyldas}), i.e.\ avoiding impredicative definitions, which Feferman continued (\cite{fefermanga,fefermandas, fefermanlight}).  Finally, the fourth `Big Five' system of Reverse Mathematics is considered the `upper limit' of predicative mathematics (See \cite{simpson2}*{\S I.12}).  
  
\smallskip

In an (similar but much more strict) effort to develop mathematics based on a predicative notion of \emph{number}, Nelson introduced \emph{predicative arithmetic} (\cite{ohnelly}).  
Unfortunately, predicative arithmetic suffers from the \emph{compatibility problem}:  
If two theorems $A, B$ are both acceptable from the point of view of predicative arithmetic, it is possible that $A\wedge B$ \emph{is not} (\cite{buss3}).  
In this light, the development of predicative arithmetic seems somewhat arbitrary.  
It is then a natural question whether Weyl-Feferman predicative mathematics suffers from the same compatibility problem.  
We show that this is the case in Section \ref{RMR2}.
A detailed discussion, also explaining our notion `acceptable in predicative mathematics', may be found in Remark~\ref{predifiel}.

\subsection{Overview of main results}\label{basic}
Our main result is that, in terms of the usual scale of comprehension axioms, a proof of the \emph{Cousin and Lindel\"of lemmas} requires a system as strong as second-order arithmetic.  
% in contrast to their countable substitutes (if existant) which live in very weak fragments of $\textsf{Z}_{2}$.  
The same result for the other theorems in Remark \ref{bthm} follows from our main result, as discussed in Section \ref{other}.  
%Our results thus subvert the `Big Five' picture of Reverse Mathematics from Section \ref{RM}.  
%We shall obtain our results by showing that it is \emph{extremely} hard to compute the objects (like a finite sub-cover from the Cousin lemma) involved in the theorems in Remark \ref{bthm}.
A precise statement of our results is found at the end of this section.    
\begin{rem}[Basic theorems]\label{bthm}\rm~
\begin{enumerate}
\renewcommand{\theenumi}{\roman{enumi}}
\item \emph{Cousin lemma}: \textbf{any} open cover of $[0,1]$ has a finite sub-cover (\cite{cousin1}).    
\item \emph{Lindel\"of lemma}: \textbf{any} open cover of $\R$ has a countable sub-cover (\cite{blindeloef}).
%\item \emph{Vitali covering theorem} as in e.g.\ \cite{bartle}*{p.\ 79}.  
\item \emph{Besicovitsch and Vitali covering lemmas} as in e.g.\ \cite{auke}*{\S2}.
\item Basic properties of the gauge integral (\cite{bartle}), like uniqueness, Hake's theorem, and extension of the Riemann integral.\label{dorfjes}  % the latter is a generalisation of the Lebesgue and the improper Riemann integral (\cite{zwette}), and provides a formalisation of Feynman's path integral.   
\item \emph{Neighbourhood Function Principle} \textsf{NFP} (\cite{troeleke1}*{p.\ 215}).  \label{NFP}
\item The existence of \emph{Lebesgue numbers} for \textbf{any} open cover (\cite{moregusto}). 
\item The \emph{Banach-Alaoglu theorem} for \textbf{any} open cover (\cite{simpson2}*{X.2.4}, \cite{xbrownphd}*{p.\ 140}).
\item The \emph{Heine-Young} and \emph{Lusin-Young} theorems, the \emph{tile theorem} \cite{YY, wildehilde}, 
and the latter's generalisation due to Rademacher (\cite{rademachen}*{p.\ 190}).
\item The Bolzano-Weierstrass, Dini, and Arzel\`a theorems \emph{for nets} (\cite{moorsmidje,ooskelly}).
%\item The paracompactness of $\R$ as in \cite{simpson2}*{II.7.2}, but for \textbf{any} cover.   
%Use the following to get Lindeloef from Paracompactness: https://math.stackexchange.com/questions/692686/countable-open-locally-finite-refinements-of-open-covers-in-paracompact-second-c
\end{enumerate}
\end{rem}
According to Bourbaki's historical note in \cite{gentop1}*{Ch.\ I}, the by far most important `acquisition' of Schoenflies' monograph \cite{schoeniswaar} 
is a theorem which constitutes a generalisation of the Cousin lemma.  
Another historical note is that Cousin (and Lindel\"of in \cite{blindeloef}*{p.\ 698}) talks about (uncountable) covers on \cite{cousin1}*{p.\ 22} as follows: 
\be\label{coukie}
\text{\emph{if for each $s\in S$ there is a circle of finite non-zero radius with $s$ as center}}
\ee
In particular, any $f:S\di \R^{+}$ gives rise to a cover \emph{in the sense of the previous quote by Cousin} as follows: $\cup_{x\in S}(x-f(x), x+f(x))$ covers $S\subset \R$.  
A rich history notwithstanding, the Cousin lemma does not show its age: there are recent attempts to develop elementary real analysis with this lemma as the `centerpiece' (\cite{thom, thom2}).

\smallskip

We now make our main results precise, for which some definitions are needed.  Detailed definitions may be found in Sections \ref{Frameingthemedia} and \ref{HCT}. 
\bdefi
Let $\textsf{Z}_{2}$ be \emph{second-order arithmetic} as defined in \cite{simpson2}*{I.2.4} and let $\Pi_{k}^{1}\textsf{-CA}_{0}$ be the fragment of $\textsf{Z}_{2}$ with comprehension restricted to $\Pi_{k}^{1}$-formulas. 
\edefi
As noted above, to formulate the theorems from the list, we require a richer language than that of $\textsf{Z}_{2}$.  
We shall make use of $\RCAo$, Kohlenbach's `base theory' of higher-order Reverse Mathematics (\cite{kohlenbach2}*{\S2}), and the associated language of \emph{all finite types}. 
We introduce this framework in detail in Section \ref{Frameingthemedia}.  
\bdefi\label{ocharme}
Let $(\exists^{3})$ state the existence of $^{3}\textbf{\textit{E}}$ from \cite{barwise}*{p.\ 713, \S12.3}; see Section \ref{HCT} for the exact definition.  
The functional $^{3}\textbf{\textit{E}}$ intuitively speaking decides any $\Pi_{\infty}^{1}$-formula.  Let $(S_{k}^{2})$ similarly state the existence of a functional deciding $\Pi_{k}^{1}$-formulas. 
For $k=1$, the subscript is omitted and $(S^{2})$ is called the \emph{Suslin functional}. 
%Let $\QFAC$ be the axiom of choice restricted to quantifier-free formulas as in \cite{kohlenbach2}*{\S1}.  
We define $\Pi_{k}^{1}\textsf{-CA}_{0}^{\omega}\equiv \RCAo+(S_{k}^{2})$, $\Z_{2}^{\omega}:=\cup_{k}\SIXK$, and $\Z_{2}^{\Omega}\equiv\RCA_{0}^{\omega}+(\exists^{3})$. 
\edefi
We discuss in Remark \ref{XXZ} why Definition \ref{ocharme} furnishes the `right' (or at least `good') higher-order analogues of the respective second-order systems. 

\smallskip

Our main results, to be proved in Section \ref{fullviewdownmainstreet1}, are now as follows: 
\begin{enumerate}
\item[(i)] The Cousin and Lindel\"of lemmas are provable in $\Z_{2}^{\Omega}$ plus a minimal fragment of the axiom of choice.
\item[(ii)] The system $\Pi_{k}^{1}\textsf{-CA}_{0}^{\omega}$ (any $k\geq 1$) cannot prove any theorem in Remark \ref{bthm}.
\item[(iii)] The Cousin lemma is \emph{equivalent} to basic properties of the gauge integral in Kohlenbach's aforementioned framework.  
\end{enumerate}
%In a nutshell, we show that \emph{full second-order arithmetic}, in the form of $(\exists^{3})$, is needed to prove the Cousin and Lindel\"of lemmas, and the same for related theorems from Remark \ref{bthm}.   
We discuss the (considerable) implications for the G\"odel hierarchy in Section \ref{kodel}.  % and \ref{}. 

\smallskip

Finally, as noted above, Reverse Mathematics is intimately connected to computability theory, and the same holds for our results;
for instance, the functional defined by $(\exists^{3})$ (resp.\ $(S_{k}^{2})$) can (resp.\ cannot) compute, in the sense of Section~\ref{HCT}, a finite sub-cover from the Cousin lemma on input an open cover of $[0,1]$.  

\smallskip
    
Our main results in Section \ref{RMR2} are then as follows: inspired by the aforementioned connection, 
we study the interaction between the theorems from the above list and the Big Five of Reverse Mathematics given by the Suslin functional $S$ and Feferman's \emph{search functional} $\mu$ from Section \ref{HCT}.  
This leads to surprising results in (higher-order) Reverse Mathematics, as follows:
\begin{enumerate}
\item The combination of the Cousin lemma and Feferman's $\mu$ yields \emph{transfinite recursion} for arithmetical formulas, i.e.\ the fourth Big Five system.  
We derive novel theorems about Borel functions from this result.    
\item The combination of the Cousin lemma and the Suslin functional $S$ yields Gandy's \emph{superjump}, the aforementioned `higher-order' \emph{Halting problem}.
\item The combination of the Lindel\"of lemma \emph{for Baire space} (resp.\ a realiser for the latter) and Feferman's $\mu$ yields $\FIVE$ (resp.\ the Suslin functional).
% the Suslin functional $S$, i.e.\ the fifth Big Five system, and by the previous item also Gandy's superjump.  
\end{enumerate}
As will become clear in Section \ref{lindeb}, the third item solves the compatibility problem of Weyl-Feferman predicativist mathematics from Section \ref{comp} in the negative.  
We also point out that the Lindel\"of lemma (resp.\ the Cousin lemma) and Feferman's $\mu$ are rather weak \emph{in isolation}, and only become strong when combined.  

\section{Preliminaries}\label{preli}
We sketch the program \emph{Reverse Mathematics} in Section \ref{RM}, discuss the associated \emph{G\"odel hierarchy} in Section \ref{kodel}, and introduce \emph{second-order and higher-order arithmetic} in Section \ref{Frameingthemedia}.  As our main results are proved using techniques from computability theory, we discuss some essential elements of the latter in Section~\ref{HCT}.

\subsection{Introducing Reverse Mathematics}\label{RM}
Reverse Mathematics (RM) is a program in the foundations of mathematics initiated around 1975 by Friedman (\cites{fried,fried2}) and developed extensively by Simpson (\cite{simpson2}) and others.  We refer to \cite{simpson2} for an overview of RM and introduce the required definitions (like the `base theory' $\RCA_{0}$) in Section \ref{RMBASE}; we now sketch some of the aspects of RM essential to this paper.  

\smallskip
  
The aim of RM is to find the axioms necessary to prove a statement of \emph{ordinary}, i.e.\ \emph{non-set theoretical}, mathematics.   
The classical base theory $\RCA_{0}$ of `computable mathematics', introduced in Section \ref{RMBASE}, is always assumed.  
Thus, the aim is:  
\begin{quote}
\emph{The aim of \emph{RM} is to find the minimal axioms $A$ such that $\RCA_{0}$ proves $ [A\di T]$ for statements $T$ of ordinary mathematics.}
\end{quote}
Surprisingly, once the minimal $A$ are known, we almost always also have $\RCA_{0}\vdash [A\asa T]$, i.e.\ we derive the theorem $T$ from the axioms $A$ (the `usual' way of doing mathematics), but we can also derive the axiom $A$ from the theorem $T$ (the `reverse' way).  In light of these `reversals', the field was baptised `Reverse Mathematics'.  

\smallskip

Perhaps even more surprisingly, in the majority of cases, for a statement $T$ of ordinary mathematics, either $T$ is provable in $\RCA_{0}$, or the latter proves $T\asa A_{i}$, where $A_{i}$ is one of the logical systems $\WKL_{0}, \ACA_{0},$ $ \ATR_{0}$ or $\FIVE$, which are all introduced in Section \ref{RMBASE}.  The latter four together with $\RCA_{0}$ form the `Big Five' and the aforementioned observation that most mathematical theorems fall into one of the Big Five categories, is called the \emph{Big Five phenomenon} (\cite{montahue}*{p.~432}).  
Furthermore, each of the Big Five has a natural formulation in terms of (Turing) computability (See \cite{simpson2}*{I}).
As noted by Simpson in \cite{simpson2}*{I.12}, each of the Big Five also corresponds (sometimes loosely) to a foundational program in mathematics.  

\smallskip

Finally, we note that the Big Five systems of RM yield a linear order:
\be\label{linord}
\FIVE\di \ATR_{0}\di \ACA_{0}\di\WKL_{0}\di \RCA_{0}.
\ee
By contrast, there are many incomparable \emph{logical} statements in second-order arithmetic.  For instance, a regular plethora of such statements may be found in the \emph{Reverse Mathematics zoo} in \cite{damirzoo}.  
The latter is intended as a collection of (somewhat natural) theorems outside of the Big Five classification of RM.  
%However, the results sketched in Section \ref{basic} \emph{fundamentally distort} the elegant picture \eqref{linord}.  
\subsection{Reverse Mathematics and the G\"odel hierarchy}\label{kodel}
The \emph{G\"odel hierarchy} is a collection of logical systems ordered via consistency strength or (essentially equivalent) inclusion\footnote{Simpson claims in \cite{sigohi}*{p.\ 112} that inclusion and consistency strength yield the same (G\"odel) hierarchy as depicted in \cite{sigohi}*{Table 1}, i.e.\ this choice does not matter.\label{fooker}}.  This hierarchy is claimed by Simpson to capture most systems that are natural and/or have foundational import, as follows. 
\begin{quote}
\emph{It is striking that a great many foundational theories are linearly ordered by $<$. Of course it is possible to construct pairs of artificial theories which are incomparable under $<$. 
However, this is not the case for the ``natural'' or non-artificial theories which are usually regarded as significant in the foundations of mathematics.} 
% The problem of explaining this observed regularity is a challenge for future foundational research.
%As an alternative to the $<$ ordering, one may consider a somewhat different ordering, the inclusion ordering [$\subset$]. [\dots] theories are always below the higher theories with respect to the $\subset$ ordering.
(\cite{sigohi})
\end{quote}
Burgess and Koelner make essentially the same claims in \cite{dontfixwhatistoobroken}*{\S1.5} and \cite{peterpeter, peterpeter2}.
However, our results imply that the theorems in Remark \ref{bthm} do not fit the G\"odel hierarchy (with the latter based on inclusion$^{\ref{fooker}}$).  
In particular, we obtain a branch that is \emph{independent} of the medium range of the G\"odel hierarchy, depicted below. 
\begin{figure}[h]
\[
\begin{array}{lll}
&\textup{\textbf{strong}} \hspace{1.5cm}& 
\left\{\begin{array}{l}
\vdots\\
\textup{supercompact cardinal}\\
\vdots\\
\textup{measurable cardinal}\\
\vdots\\
\ZFC \\
\textsf{\textup{ZC}} \\
\textup{simple type theory}
\end{array}\right.
\\
&& \\
  &&~\quad{ {\Z_{2}^{\Omega}+\QFAC^{0,1}}}\\
&&\\
&\textup{\textbf{medium}} & 
\left\{\begin{array}{l}
 \Z_{2}^{\omega} \equiv \cup_{k}\SIXK\\
\vdots\\
\textup{$\Pi_{2}^{1}\textsf{-CA}_{0}^{ {\omega}}$}\\
\textup{$\FIVE^{ {\omega}}$ }\\
\textup{$\ATR_{0}^{ {\omega}}$}  \\
\textup{$\ACA_{0}^{ {\omega}}$} \\
\end{array}\right.
%\begin{array}{c}
%\textup{Kohlenbach's}\\
%\textup{ {higher-order RM}}\\
%\end{array}
\\
&
\\
{ {\left\{\begin{array}{l}
\textup{Cousin and Lindel\"of lemmas}\\
\textup{basic prop.\ of gauge integral}\\
\end{array}\right\}}}
&\begin{array}{c}\\\textup{\textbf{weak}}\\ \end{array}& 
\left\{\begin{array}{l}
\WKL_{0}^{ {\omega}} \\
\textup{$\RCA_{0}^{ {\omega}}$} \\
\textup{$\textsf{PRA}$} \\
\textup{$\textsf{EFA}$ } \\
\textup{bounded arithmetic} \\
\end{array}\right.
\\
\end{array}
\]
\caption{The G\"odel hierarchy with a side-branch for the medium range}\label{xxy}
\begin{picture}(250,0)
\put(155,200){ {\vector(-3,-2){125}}}
\put(160,117){ {\vector(-3,-2){53}}}
\put(100,70){ {\vector(1,0){80}}}
\put(125,100){ {\vector(3,2){50}}}
\put(150,100){{\line(-5,3){20}}}
\end{picture}
\end{figure}
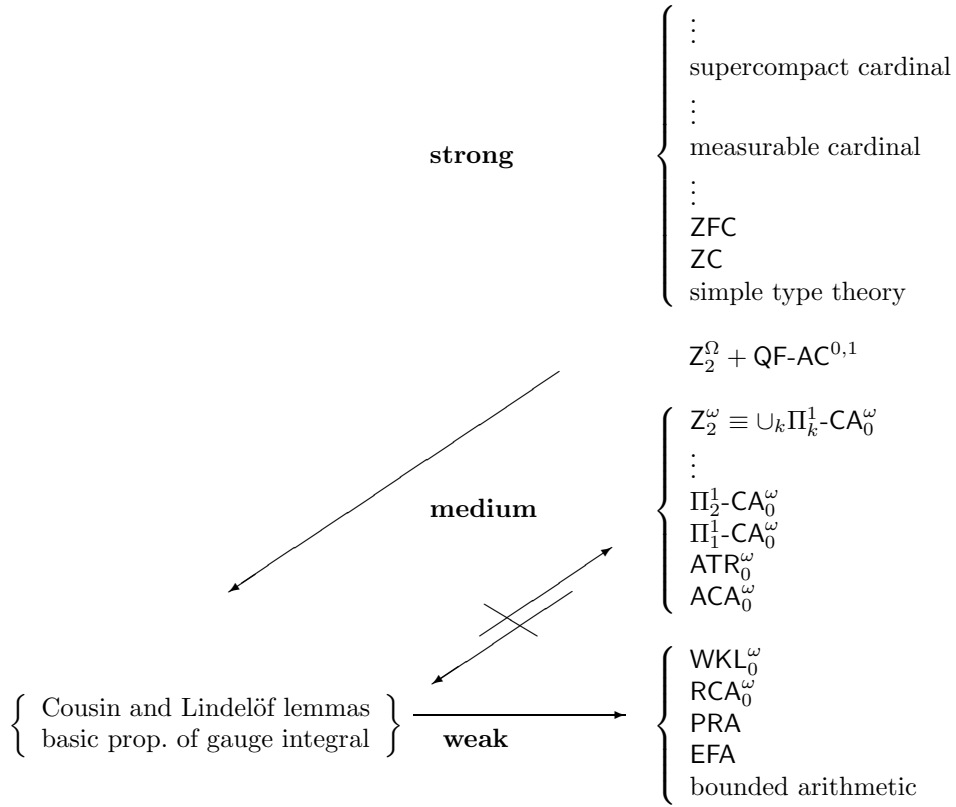\\
Arguably, the G\"odel hierarchy is a central object of study in mathematical logic, as e.g.\ stated by Simpson in \cite{sigohi}*{p.\ 112} or Burgess in \cite{dontfixwhatistoobroken}*{p.\ 40}.  
Some remarks on the technical details concerning Figure \ref{xxy} are as follows. 
\begin{enumerate}
\item Note that we use a \emph{non-essential} modification of the G\"odel hierarchy, namely involving systems of higher-order arithmetic, like e.g.\ $\ACA_{0}^{\omega}$ instead of $\ACA_{0}$; these systems are (at least) $\Pi_{2}^{1}$-conservative over the associated second-order system (See e.g.\ \cite{yamayamaharehare}*{Theorem 2.2}).  
\item In the spirit of RM, we show in \cite{dagsamV} that the Cousin lemma and (a version of) the Lindel\"of lemma are provable \emph{without} the use of $\QFAC^{0,1}$, as also discussed in Remark \ref{linpinpon}.     
\item The system $\Z_{2}^{\Omega}$ is placed \emph{between} the medium and strong range, as the combination of the recursor $\textsf{R}_{2}$ from G\"odel's $T$ and $\exists^{3}$ yields a system stronger than $\Z_{2}^{\Omega}$.  The system $\SIXK$ is more robust in this sense.    
\item Despite Simpson's grand claim in the above quote, there are now some examples of logical systems that fall outside of the G\"odel hierarchy, like \emph{special cases} of Ramsey's theorem and the axiom of determinacy (\cites{dsliceke, shoma}).    
\end{enumerate}
Finally, in light of the equivalences involving the gauge integral and the Cousin lemma (See Section \ref{introgau}), the latter seriously challenges the `Big Five' classification from RM,  the linear nature of the G\"odel hierarchy,
%NEW&&
 as well as Feferman's claim that the mathematics necessary for the development of physics can be formalised in relatively weak logical systems (See Remark \ref{fefferketoch} for the latter claim).

\subsection{The framework of Reverse Mathematics}\label{Frameingthemedia}
We introduce axiomatic systems essential to RM.  
We start with a sketch of \emph{second-order arithmetic} (See \cite{simpson2}*{I.2.4}), the framework of Friedman-Simpson RM, and finish with \emph{higher-order artihmetic}, the framework of Kohlenbach's \emph{higher-order} RM (See \cite{kohlenbach2}).  
\subsubsection{Second-order arithmetic and fragments}\label{RMBASE}
The language $\L_{2}$ of second-order arithmetic $\Z_{2}$ has two sorts of variables: \emph{number variables} $n, m, k, l, \dots$ intended to range over the natural numbers, and \emph{set variables} $X, Y, Z, \dots$ intended to range over sets of natural numbers.  
The constants of $\L_{2}$ are $0, 1, <_{\N}, +_{\N}, \times_{\N}, =_{\N}$ and $ \in$, which are intended to have their usual meaning (by the axioms introduced below).  Formulas and terms are built up from these constants in the usual way.  
\bdefi Second-order arithmetic $\Z_{2}$ consists of three axiom schemas:
\begin{enumerate}
\item Basic axioms expressing that $0, 1, <_{\N}, +_{\N}, \times_{\N}$ form an ordered semi-ring with equality $=_{\N}$.  
\item Induction: For any $X$, $\big(0\in X\wedge (\forall n)(n\in X\di n+1\in X)\big)\di (\forall n)(n\in X)$.
\item Comprehension: For any formula $\varphi(n)$ of $\L_{2}$ which does not involve the variable $X$, we have $(\exists X)(\forall n)(n\in X\asa \varphi(n))$.
\end{enumerate}
\edefi
Induction is well-known, while comprehension intuitively expresses 
that any $\L_{2}$-formula $\varphi(n)$ yields a set $X=\{n\in \N:\varphi(n)\}$ consisting of exactly those numbers $n\in \N$ satisfying $\varphi(n)$.  
Now, \emph{fragments} of $\Z_{2}$ are obtained by restricting comprehension (and induction), for which the following definition is needed.  
\bdefi[Formula classes]\label{ahah}
\begin{enumerate}
\item A formula of $\L_{2}$ is \emph{quantifier-free} ($\Sigma_{0}^{0}$ or $\Pi_{0}^{0}$) if it does not involve quantifiers.  
To be clear: variables are allowed; only quantifiers are banned.  
\item A formula of $\L_{2}$ is \emph{arithmetical} ($\Sigma_{0}^{1}$ or $\Pi_{0}^{1}$) if it only involves quantifiers over number variables, i.e.\ set quantifiers like $(\exists X)$ and $(\forall Y)$ are not allowed.  
\item An arithmetical formula is $\Sigma_{k+1}^{0}$ (resp.\ $\Pi_{k+1}^{0}$) if it has the form $(\exists n)\varphi(n)$ (resp.\ $(\forall n)\varphi(n)$) with $\varphi$ in $\Pi_{k}^{0}$ (resp.\ in $\Sigma_{k}^{0}$).
\item A formula of $\L_{2}$ is $\Sigma_{k+1}^{1}$ (resp.\ $\Pi_{k+1}^{1}$) if it has the form $(\exists X)\varphi(X)$ (resp.\ $(\forall X)\varphi(X)$) with $\varphi$ in $\Pi_{k}^{1}$ (resp.\ in $\Sigma_{k}^{1}$).
\item A formula of $\L_{2}$ is $\Delta^{i}_{k+1}$ if it is both $\Pi_{k+1}^{i}$ and $\Sigma_{k+1}^{i}$ for $i=0,1$.
%\item A formula of $\L_{2}$ is $\Delta^{1}_{k+1}$ if it is both $\Pi_{k+1}^{1}$ and $\Sigma_{k+1}^{1}$.
\end{enumerate}
\edefi
Intuitively, a $\Sigma_{k}^{0}$-formula is a quantifier-free formula pre-fixed by $k$ alternating number quantifiers, starting with an existential one; a $\Sigma_{k}^{1}$-formula is an arithmetical formula pre-fixed by $k$ alternating set quantifiers, starting with an existential one.  
The $\Pi$-formulas are (equivalent to) negations of the corresponding $\Sigma$-versions.  

\smallskip

Using the above, the third and fifth `Big Five' systems $\ACA_{0}$ and $\FIVE$ are just $\Z_{2}$ with comprehension restricted to resp.\ arithmetical and $\Pi_{1}^{1}$-formulas.  
Alternatively, $\ACA_{0}$ allows one to build sets using \emph{finite} iterations of Turing's \emph{Halting problem} (\cite{tur37}), aka the \emph{Turing jump}; intuitively, $\ATR_{0}$ extends this to \emph{transfinite recursion}, i.e.\ the \emph{unbounded} iteration of the Turing jump \emph{along any countable well-ordering}.  
Furthermore, the `base theory' $\RCA_{0}$ is $\Z_{2}$ with comprehension restricted to $\Delta_{1}^{0}$-formulas, plus induction for $\Sigma_{1}^{0}$-formulas.  
As discussed in \cite{simpson2}*{II and IX.3}, $\Delta_{1}^{0}$-comprehension essentially expresses that `all computable sets exists', while $\Sigma_{1}^{0}$-induction corresponds to
primitive recursion in the sense of Hilbert's \emph{finitistic mathematics}.  The system $\WKL_{0}$ is just $\RCA_{0}$ extended by the \emph{weak K\"onig's lemma}\footnote{To be absolutely clear, we take `$\WKL$' to be the $\L_{2}$-sentence \emph{every infinite binary tree has a path} as in \cite{simpson2}, while the Big Five system $\WKL_{0}$ is $\RCA_{0}+\WKL$, and $\WKL_{0}^{\omega}$ is $\RCAo+\WKL$.} ($\WKL$ hereafter) which states that an infinite binary tree has a path.  

\smallskip

Finally, in light of the previous and \eqref{linord}, the Big Five only constitute a \emph{very tiny fragment} of $\Z_{2}$; on a related note, the RM of topology does give rise to theorems equivalent to $\SIX$ (\cite{mummy}), but that is the current upper bound of RM to the best of our knowledge.  
In particular, if $\SIXk$ is $\Z_{2}$ restricted to $\Pi_{k}^{1}$-comprehension, then this system can be said to `go beyond Friedman-Simpson RM' for $k\geq 3$.

\subsubsection{Higher-order arithmetic and fragments}\label{KOH}
As suggested by its name, \emph{higher-order arithmetic} extends second-order arithmetic.  Indeed, while the latter is restricted to numbers and sets of numbers, higher-order arithmetic also has sets of sets of numbers, sets of sets of sets of numbers, et cetera.  
To formalise this idea, we introduce the collection of \emph{all finite types} $\mathbf{T}$, defined by the two clauses:
\begin{center}
(i) $0\in \mathbf{T}$   and   (ii)  If $\sigma, \tau\in \mathbf{T}$ then $( \sigma \di \tau) \in \mathbf{T}$,
\end{center}
where $0$ is the type of natural numbers, and $\sigma\di \tau$ is the type of mappings from objects of type $\sigma$ to objects of type $\tau$.
In this way, $1\equiv 0\di 0$ is the type of functions from numbers to numbers, and where  $n+1\equiv n\di 0$.  Viewing sets as given by their characteristic function, we note that $\Z_{2}$ only includes objects of type $0$ and $1$.    

\smallskip

The language of $\L_{\omega}$ consists of variables $x^{\rho}, y^{\rho}, z^{\rho},\dots$ of any finite type $\rho\in \mathbf{T}$.  Types may be omitted when they can be inferred from context.  
The constants of $\L_{\omega}$ includes the type $0$ objects $0, 1$ and $ <_{0}, +_{0}, \times_{0},=_{0}$  which are intended to have the same meaning as their $\N$-subscript counterparts in $\Z_{2}$.
Equality at higher types is defined in terms of `$=_{0}$' as follows: for any objects $x^{\tau}, y^{\tau}$, we have
\be\label{aparth}
[x=_{\tau}y] \equiv (\forall z_{1}^{\tau_{1}}\dots z_{k}^{\tau_{k}})[xz_{1}\dots z_{k}=_{0}yz_{1}\dots z_{k}],
\ee
if the type $\tau$ is composed as $\tau\equiv(\tau_{1}\di \dots\di \tau_{k}\di 0)$.  
Furthermore, $\L_{\omega}$ also includes the \emph{recursor constant} $\mathbf{R}_{\sigma}$ for any $\sigma\in \mathbf{T}$, which allows for iteration on type $\sigma$-objects as in the special case \eqref{special}.  
Formulas and terms are defined as usual.  
\bdefi The base theory $\RCAo$ consists of the following axioms:
\begin{enumerate}
\item  Basic axioms expressing that $0, 1, <_{0}, +_{0}, \times_{0}$ form an ordered semi-ring with equality $=_{0}$.
\item Basic axioms defining the well-known $\Pi$ and $\Sigma$ combinators (aka $K$ and $S$ in \cite{avi2}), which allow for the definition of \emph{$\lambda$-abstraction}. 
\item The defining axiom of the recursor constant $\mathbf{R}_{0}$: For $m^{0}$ and $f^{1}$: 
\be\label{special}
\mathbf{R}_{0}(f, m, 0):= m \textup{ and } \mathbf{R}_{0}(f, m, n+1):= f( \mathbf{R}_{0}(f, m, n)).
\ee
\item The \emph{axiom of extensionality}: for all $\rho, \tau\in \mathbf{T}$, we have:
\be\label{EXT}\tag{$\textsf{\textup{E}}_{\rho, \tau}$}  
(\forall  x^{\rho},y^{\rho}, \varphi^{\rho\di \tau}) \big[x=_{\rho} y \di \varphi(x)=_{\tau}\varphi(y)   \big].
\ee 
\item The induction axiom for quantifier-free\footnote{To be absolutely clear, similar to Definition \ref{ahah}, variables (of any finite type) are allowed in quantifier-free formulas: only quantifiers are banned.} formulas.
\item $\QFAC^{1,0}$: The quantifier-free axiom of choice as in Definition \ref{QFAC}.
\end{enumerate}
\edefi
\bdefi\label{QFAC} The axiom $\QFAC$ consists of the following for all $\sigma, \tau \in \textbf{T}$:
\be\tag{$\QFAC^{\sigma,\tau}$}
(\forall x^{\sigma})(\exists y^{\tau})A(x, y)\di (\exists Y^{\sigma\di \tau})(\forall x^{\sigma})A(x, Y(x))
\ee
for any quantifier-free formula $A$ in the language of $\L_{\omega}$.
\edefi
As discussed in \cite{kohlenbach2}*{\S2}, $\RCAo$ and $\RCA_{0}$ prove the same sentences `up to language' as the latter is set-based and the former function-based.  

\smallskip

Furthermore, recursion as in \eqref{special} is called \emph{primitive recursion}; the class of functionals obtained from $\mathbf{R}_{\rho}$ for all $\rho \in \mathbf{T}$ is called \emph{G\"odel's system $T$} of all (higher-order) primitive recursive functionals.  

\smallskip

We use the usual notations for natural, rational, and real numbers, and the associated functions, as introduced in \cite{kohlenbach2}*{p.\ 288-289}.  
\begin{defi}[Real numbers and related notions in $\RCAo$]\label{keepintireal}\rm~
\begin{enumerate}
\renewcommand{\theenumi}{\roman{enumi}}
\item Natural numbers correspond to type zero objects, and we use `$n^{0}$' and `$n\in \N$' interchangeably.  Rational numbers are defined as signed quotients of natural numbers, and `$q\in \Q$' and `$<_{\Q}$' have their usual meaning.    
\item Real numbers are coded by fast-converging Cauchy sequences $q_{(\cdot)}:\N\di \Q$, i.e.\  such that $(\forall n^{0}, i^{0})(|q_{n}-q_{n+i})|<_{\Q} \frac{1}{2^{n}})$.  
We use Kohlenbach's `hat function' from \cite{kohlenbach2}*{p.\ 289} to guarantee that every $f^{1}$ defines a real number.  
\item We write `$x\in \R$' to express that $x^{1}:=(q^{1}_{(\cdot)})$ represents a real as in the previous item and write $[x](k):=q_{k}$ for the $k$-th approximation of $x$.    
\item Two reals $x, y$ represented by $q_{(\cdot)}$ and $r_{(\cdot)}$ are \emph{equal}, denoted $x=_{\R}y$, if $(\forall n^{0})(|q_{n}-r_{n}|\leq \frac{1}{2^{n-1}})$. Inequality `$<_{\R}$' is defined similarly.         
\item Functions $F:\R\di \R$ mapping reals to reals are represented by $\Phi^{1\di 1}$ mapping equal reals to equal reals, i.e. $(\forall x, y)(x=_{\R}y\di \Phi(x)=_{\R}\Phi(y))$.\label{REXTJE}
\item The relation `$x\leq_{\tau}y$' is defined as in \eqref{aparth} but with `$\leq_{0}$' instead of `$=_{0}$'.  
\item Sets of type $\rho$ objects $X^{\rho\di 0}, Y^{\rho\di 0}, \dots$ are given by their characteristic functions $f^{\rho\di 0}_{X}$, i.e.\ $(\forall x^{\rho})[x\in X\asa f_{X}(x)=_{0}1]$, where $f_{X}^{\rho\di 0}\leq_{\rho\di 0}1$.  
\end{enumerate}
\end{defi}
We sometimes omit the subscript `$\R$' if it is clear from context.  
We also introduce some notation to handle finite sequences nicely.  
\begin{nota}[Finite sequences]\label{skim}\rm
We assume a dedicated type for `finite sequences of objects of type $\rho$', namely $\rho^{*}$.  Since the usual coding of pairs of numbers goes through in $\RCAo$, we shall not always distinguish between $0$ and $0^{*}$. 
Similarly, we do not always distinguish between `$s^{\rho}$' and `$\langle s^{\rho}\rangle$', where the former is `the object $s$ of type $\rho$', and the latter is `the sequence of type $\rho^{*}$ with only element $s^{\rho}$'.  The empty sequence for the type $\rho^{*}$ is denoted by `$\langle \rangle_{\rho}$', usually with the typing omitted.  

\smallskip

Furthermore, we denote by `$|s|=n$' the length of the finite sequence $s^{\rho^{*}}=\langle s_{0}^{\rho},s_{1}^{\rho},\dots,s_{n-1}^{\rho}\rangle$, where $|\langle\rangle|=0$, i.e.\ the empty sequence has length zero.
For sequences $s^{\rho^{*}}, t^{\rho^{*}}$, we denote by `$s*t$' the concatenation of $s$ and $t$, i.e.\ $(s*t)(i)=s(i)$ for $i<|s|$ and $(s*t)(j)=t(|s|-j)$ for $|s|\leq j< |s|+|t|$. For a sequence $s^{\rho^{*}}$, we define $\overline{s}N:=\langle s(0), s(1), \dots,  s(N-1)\rangle $ for $N^{0}<|s|$.  
For a sequence $\alpha^{0\di \rho}$, we also write $\overline{\alpha}N=\langle \alpha(0), \alpha(1),\dots, \alpha(N-1)\rangle$ for \emph{any} $N^{0}$.  By way of shorthand, 
$(\forall q^{\rho}\in Q^{\rho^{*}})A(q)$ abbreviates $(\forall i^{0}<|Q|)A(Q(i))$, which is (equivalent to) quantifier-free if $A$ is.   
\end{nota}

\subsection{Higher-order computability theory}\label{HCT}
As noted above, our main results will be proved using techniques from computability theory.
Thus, we first make our notion of `computability' precise as follows.  
\begin{enumerate}
\item[(I)] We adopt $\ZFC$, i.e.\ Zermelo-Fraenkel set theory with the Axiom of Choice, as the official metatheory for all results, unless explicitly stated otherwise.
\item[(II)] We adopt Kleene's notion of \emph{higher-order computation} as given by his nine clauses S1-S9 (See \cites{longmann, Sacks.high}) as our official notion of `computable'.
\end{enumerate}
%Some of our results require familiarity with computability theory as in the second item.  
For the rest of this section, we introduce some existing functionals which will be used below.
These functionals constitute the counterparts of $\Z_{2}$, and some of the Big Five systems, in higher-order RM.
First of all, $\ACA_{0}$ is readily derived from:
\be\label{mu}\tag{$\mu^{2}$}
(\exists \mu^{2})\big[ (\forall f^{1})( (\exists n)(f(n)=0 )\di f(\mu(f))=0)    \big], 
\ee
and $\ACA_{0}^{\omega}\equiv\RCAo+(\mu^{2})$ proves the same $\Pi_{2}^{1}$-sentences as $\ACA_{0}$ by \cite{yamayamaharehare}*{Theorem~2.2}.   The functional $\mu^{2}$ in $(\mu^{2})$ is also called \emph{Feferman's $\mu$} (\cite{avi2}), 
and is clearly \emph{discontinuous} at $f=_{1}11\dots$; in fact, $(\mu^{2})$ is equivalent to the existence of $F:\R\di\R$ such that $F(x)=1$ if $x>_{\R}0$, and $0$ otherwise (\cite{kohlenbach2}*{\S3}).

\smallskip
\noindent
Secondly, $\FIVE$ is readily derived from the following sentence:
\be\tag{$S^{2}$}
(\exists S^{2}\leq_{2}1)(\forall f^{1})\big[  (\exists g^{1})(\forall x^{0})(f(\overline{g}n)=0)\asa S(f)=0  \big], 
\ee
and $\FIVE^{\omega}\equiv \RCAo+(S^{2})$ proves the same $\Pi_{3}^{1}$-sentences as $\FIVE$ by \cite{yamayamaharehare}*{Theorem 2.2}.   The functional $S^{2}$ in $(S^{2})$ is also called \emph{the Suslin functional} (\cite{kohlenbach2}).
By definition, the Suslin functional $S^{2}$ can decide whether a $\Sigma_{1}^{1}$-formula (as in the left-hand side of $(S^{2})$) is true or false.   We similarly define the functional $S_{k}^{2}$ which decides the truth or falsity of $\Sigma_{k}^{1}$-formulas; we also define 
the system $\SIXK$ as $\RCAo+(S_{k}^{2})$, where  $(S_{k}^{2})$ expresses that the functional $S_{k}^{2}$ exists.

\smallskip

\noindent
Thirdly, full second-order arithmetic $\Z_{2}$ is readily derived from the sentence:
\be\tag{$\exists^{3}$}
(\exists E^{3}\leq_{3}1)(\forall Y^{2})\big[  (\exists f^{1})Y(f)=0\asa E(Y)=0  \big], 
\ee
and we define $\Z_{2}^{\Omega}\equiv \RCAo+(\exists^{3})$, which is a conservative extension of $\Z_{2}$ by \cite{hunterphd}*{Cor.~2.6}.   The (unique) functional from $(\exists^{3})$ is also called `$\exists^{3}$', and we will use a similar convention for other functionals.  

\smallskip
\noindent
Fourth, there is primitive recursive function $U$ such that `$U(e,k,n)=_{0}m+1$' expresses that the $e$-th Turing machine with input $k$ halts after $n$ steps with output $m$.  
By definition, Feferman's $\mu^{2}$ provides an upper bound on this $n$ \emph{if it exists}, i.e.\ we can use $\mu^{2}$ to solve the Halting problem.  Similarly, Gandy's \emph{superjump} solves the Halting problem for 
higher-order computability as follows: 
\be\tag{$\SJ^{3}$}
\SJ(F^{2},e^{0}):=
\begin{cases}
0 & \textup{ if $\{e\}(F)$ terminates}\\
1 & \textup{otherwise}
\end{cases},
\ee where $e$ is an S1-S9-index.
A characterisation of $\SJ$ in terms of discontinuities may be found in \cite{hartjeS}.  %Hence, there are a number of similarities between $\mu^{2}$ and $\SJ$.
Clearly, the above functionals are natural counterparts of (set-based) comprehension axioms in a functional-based language. 

\smallskip

Fifth, recall that the Cousin lemma from Remark \ref{bthm} states the existence of a finite sub-cover for an open cover of the unit interval. 
Since Cantor space is homeomorphic to a closed subset of $[0,1]$, the former inherits the same property.  
In particular, for any $G^{2}$, the corresponding `canonical cover' of $2^{\N}$ is $\cup_{f\in 2^{\N}}[\overline{f}G(f)]$ where $[\sigma^{0^{*}}]$ is the set of all binary extensions of $\sigma$.  By compactness, there is a finite sequence $\langle f_0 , \ldots , f_n\rangle$ such that the set of $\cup_{i\leq n}[\bar f_{i} G(f_i)]$ still covers $2^{\N}$.  
We now introduce the specification $\SCF(\Theta)$ for a (non-unique) functional $\Theta$ which computes such a finite sequence.  
We refer to such a functional $\Theta$ as a \emph{realiser} for the compactness of Cantor space, and simplify its type to `$3$' to improve readability.
\bdefi\label{dodier}
The formula $\SCF(\Theta)$ is as follows for $\Theta^{2\di 1^{*}}$:
\be\label{normaal}
(\forall G^{2})(\forall f^{1}\leq_{1}1)(\exists g\in \Theta(G))(f\in [\overline{g}G(g)]).
\ee
where `$f\in [\overline{g}G(g)]$' is the quantifier-free formula $\overline{f}G(g)=_{0^{*}}\overline{g}G(g)$.
\edefi
Clearly, there is no unique $\Theta$ as in \eqref{normaal} (just add more binary sequences to $\Theta(G)$); nonetheless, 
we have in the past referred to any $\Theta$ satisfying $\SCF(\Theta)$ as `the' \emph{special fan functional} $\Theta$, and we will continue this abuse of language.  
We shall however repeatedly point out the non-unique nature of the special fan functional $\Theta$ in the following.   While $\Theta$ may appear exotic at first, it provides the only method we can think of for computing gauge integrals \emph{in general}, as discussed in Remark~\ref{engauged}.   
As to its provenance, $\Theta$ was introduced as part of the study of the \emph{Gandy-Hyland functional} in \cite{samGH}*{\S2} via a slightly different definition.  
These definitions are identical up to a term of G\"odel's $T$ of low complexity.  

\smallskip

Finally, we should discuss why the above systems involving the `$\omega$' superscripts are the `right' (or at least `good') higher-order analogues of the correspoding second-order systems.
We also discuss the special case of $\Z_{2}^{\Omega}$ and second-order arithmetic.    
\begin{rem}\label{XXZ}\rm
First of all, Kohlenbach introduces $\RCAo$ in \cite{kohlenbach2} as the base theory for higher-order RM and proves that it is conservative over $\RCA_{0}$ up to language.    
Hence, it makes sense to similarly use the superscript `$\omega$' to denote the higher-order counterparts of subsystems of second-order arithmetic $\Z_{2}$

\smallskip

Secondly, most of the aforementioned systems with superscript `$\omega$' are known conservative extensions (for at least $\Pi_{2}^{1}$-formulas) of their second-order counterparts. 
For $\RCAo$, this follows from \cite{kohlenbach2}*{Prop.\ 3.1}.
For $\ACAo$ and $\FIVE^{\omega}$, this follows from \cite{yamayamaharehare}*{Theorem 2.2}, while for $\Z_{2}^{\omega}$ and $\Z_{2}^{\Omega}$ this follows from \cite{hunterphd}*{Cor.\ 2.6}. 
Similar results for $\SIXK$ can be obtained in the same way.  

\smallskip

Thirdly, as noted below Figure \ref{xxy} in Section \ref{kodel}, $\Z_{2}^{\Omega}$ is placed \emph{between} the medium and strong range.  The motivation is that the combination of the recursor $\textsf{R}_{2}$ from G\"odel's $T$ and $\exists^{3}$ yields a system stronger than $\Z_{2}^{\Omega}$.   
On the other hand, the system $\Z_{2}^{\omega}$ does not suffer from this problem, and we therefore believe that the latter is the `right' higher-order analogue of second-order arithmetic $\Z_{2}$.
\end{rem}
\section{Main results I}\label{fullviewdownmainstreet1}
We establish our main results as sketched in Section \ref{basic}.   % using techniques from higher-order computability theory.  
We treat the Cousin lemma in full detail in Section \ref{CL}, while similar `covering theorems' from Remark~\ref{bthm} 
are treated analogously in Section \ref{korkske}.  We show in Section \ref{introgau} that the Cousin lemma is equivalent to various basic properties of the \emph{gauge integral}.
In Section \ref{klipel}, we derive the Cousin lemma from the following generalisation of the Bolzano-Weierstrass theorem: \emph{every net in the unit interval has a convergent sub-net}. 
Nets (aka Moore-Smith sequences) provide a generalisation of the concept of sequences beyond countable index sets, going back a century (\cites{moorsmidje,moorelimit2}).

\subsection{Cousin lemma}\label{CL}
Cousin first proved (what is now known as) the \emph{Cousin lemma} before 1893 (\cite{dugac1}).  
This lemma essentially expresses that $I=[0,1]$ is Heine-Borel compact, i.e.\ that any open cover of $I$ has a finite sub-cover.  
The goal of this section is to establish that, despite its seemingly elementary nature, the Cousin lemma can only be proved in full second-order arithmetic, as sketched in Section~\ref{basic}.  
This should be contrasted with the restriction to \emph{countable} covers, which may be proved in the weak fragment $\WKL_{0}$ by \cite{simpson2}*{IV.1.2}).

\smallskip

First of all, a functional $\Psi:\R\di \R^{+}$ gives rise to the (uncountable) \emph{canonical} open cover $\cup_{x\in I} I_{x}^{\Psi}$ where $I_{x}^{\Psi}$ is the open interval $(x-\Psi(x), x+\Psi(x))$.  
Hence, the Cousin lemma implies that $\cup_{x\in I} I_{x}^{\Psi}$ has a finite sub-cover; in symbols:
\be\tag{$\HBU$}
(\forall \Psi:\R\di \R^{+})(\exists \langle y_{1}, \dots, y_{k}\rangle)(\forall x\in I)(\exists i\leq k)(x\in I_{y_{i}}^{\Psi}).
\ee
Note that $\HBU$ makes use of the original formulation by Cousin as in \eqref{coukie}.  We show in \cite{sahotop}*{\S3.4} that $\HBU$ sports a certain robustness, in that its logical properties do not depend on the exact choice of definition of cover.  

\smallskip

The main goal of this section is to prove the following theorem, which establishes that full 
second-order arithmetic is needed to prove the Cousin lemma as in $\HBU$. 
\begin{thm}\label{main1}
 $\Z_{2}^{\Omega}+\QFAC^{0,1}$ proves $\HBU$; no system $\SIXK$ proves $\HBU$.
\end{thm}
The first part is a necessity as otherwise the designation ``analysis'' for $\Z_{2}$ would be meaningless (\cite{overwinning}*{p.~291}).  
The second part constitutes a surprise: the restriction of $\HBU$ to \emph{countable} covers is equivalent to $\WKL_{0}$ (\cite{simpson2}*{IV.1}), a system with the (first-order) strength of $\RCAo$. Kohlenbach has introduced generalisations of $\WKL_{0}$ with properties similar to $\HBU$ (\cite{kohlenbach4}*{\S5-6}), but these axioms \emph{do not stem from mathematics}, i.e.\ are `purely logical'.  Furthermore, $\HBU$ is \emph{robust} (\cite{montahue}*{p.\ 432}) in that restricting the variable $x$ to the (Turing) computable reals or the rationals in $I$ does not make a difference.       
We now prove the first part of Theorem~\ref{main1}.
\begin{thm}\label{main1a}
The system $\Z_{2}^{\Omega}+\QFAC^{0,1}$ proves $\HBU$.
\end{thm}
\begin{proof}
We only sketch the proof as it makes use of items from Remark \ref{bthm} to be studied in Section \ref{korkske}.
A full proof may be found in Theorem \ref{main2}.
Now, to derive $\HBU$, we note that the \emph{Lindel\"of lemma} provides a \emph{countable} sub-cover for any open cover of $I$.  Since $(\exists^{3})$ immediately implies $\Z_{2}$, we may use \cite{simpson2}*{IV.1.2}, which 
implies that every countable open cover has a finite sub-cover.  What remains is to prove the Lindel\"of lemma, which readily follows from the \emph{Neighbourhood function principle} \textsf{NFP}, i.e.\ item \eqref{NFP} in Remark \ref{bthm}, as will become clear in the proof of Theorem \ref{main2}.  
In turn, \textsf{NFP} has a straightforward proof in $\Z_{2}^{\Omega}+\QFAC^{0,1}$, as will also become clear in the proof of Theorem \ref{main2}.  
\end{proof}
As noted above, we shall make use of computability theory to establish Theorem~\ref{main1}.  
Hence, we first show that $\HBU$ is equivalent to the existence of the special fan functional $\Theta$ in Theorem \ref{nolapdog}.  Theorem \ref{main1} will then be established
by showing that models of $\SIXK$ do not always contain $\Theta$ as in Theorem \ref{realmain1}.  Note that the functional $\Omega$ as in \eqref{1337} is called a \emph{realiser} for $\HBU$.
\begin{thm}\label{nolapdog}
$\ACAo+\QFAC^{2,1}$ proves $(\exists \Theta)\SCF(\Theta)\asa \HBU \asa \eqref{1337}$, where
\be\label{1337}
(\exists \Omega^{2\di 1^{*}})(\forall \Psi:\R\di \R^{+})(\forall x\in [0,1])(\exists y\in \Omega(\Psi)(x\in I_{y}^{\Psi}).
\ee
\end{thm}
\begin{proof}
We first point out two useful properties of Feferman's $\mu$: the axiom $(\mu^{2})$ defining the latter functional is equivalent to the existence of $F:\R\di\R$ such that $F(x)=1$ if $x>_{\R}0$, and $0$ otherwise (\cite{kohlenbach2}*{\S3}).  Furthermore, by repeatedly applying $\mu$, we can show that any arithmetical formula is equivalent to a quantifier-free one.  We also recall the notation `$f\in [\sigma]$' for covers of Cantor space from Definition \ref{dodier}.  

\smallskip 
 
Based on the previous, given $\Psi,y_1 , \ldots , y_k$ as in $\HBU$, we can decide if the intervals $I^\Psi_{y_i}$ form an open covering or not: we just check (using $\mu$) how the end-points of these intervals are interleaved.  Thus, using $\mu$ as a parameter, we can deduce \eqref{1337} from $\HBU$ by $\QFAC^{2,1}$. Likewise, given $f_1 , \ldots , f_n\leq_{1}1$ and $k_1 , \ldots ,k_n$ in $\N$, we can decide if the set of neighbourhoods $[\bar f_ik_i]$ form a covering or not; hence, we may use $\QFAC^{2,1}$ to similarly obtain $\Theta$ from the compactness of Cantor space.

\smallskip

Now define $\xi(f) = \sum_{i \in \N}f(i)\cdot 2^{-(i+1)}$ and $\zeta(f) = \sum_{i \in \N}2f(i)\cdot 3^{-(i+1)}$ for $f\in \{0,1\}^{\N}$; note that $\xi$ is a continuous projection of $\{0,1\}^\N$ to $[0,1]$, while $\zeta$ is the homeomorphism between $\{0,1\}^\N$ and the classical Cantor space $C^c$.  Using $\xi$ and $\zeta$, we can convert canonical covers between $I$ and Cantor space as follows:  
\begin{itemize}
\item For $\Psi:[0,1] \rightarrow \R^+$, define $F_\Psi(f)$ as the least $n$ such that $[\bar f n] \subseteq \xi^{-1}(I_{\xi(f)}^\Psi)$.
\item For $F:\{0,1\}^\N \rightarrow \N$, we define $\Psi_F(x)$ as the distance from $x$ to $C^c$ if $x \not \in C^c$, and as the least rational (in some canonical enumeration of $\Q^+$) $q$ such that $\zeta^{-1}((x - q,x+q)) \subseteq [\overline{\zeta^{-1}(x)}F(\zeta^{-1}(x))]$ if $x \in C^c$.
\end{itemize}
These constructions are arithmetical, and the compactness property for the associated coverings are transferred from one space to the other in both directions.
\end{proof}
From the proof, we may also conclude that there is a term $t$ such that if $\SCF(\Theta)$ and $\Omega := t(\Theta,\mu)$ then $\Omega$ satisfies \eqref{1337}, and conversely, there is a term $s$ such that if $\Omega$ satisfies \eqref{1337} and $\Theta := s(\Omega,\mu)$, then $\SCF(\Theta)$. 
The proof makes use of the \emph{Axiom of Choice} (as in $\QFAC$) to obtain a functional $\Theta$ as in $\SCF(\Theta)$, resp. $\Omega$ satisfying \eqref{1337}, from the existence of finite sub-coverings.  
Nonetheless, a careful analysis of known proofs of $\HBU$ yields such  functionals $\Theta$ and $\Omega$ \emph{without the Axiom of Choice}.
We discuss this in more detail in Remark \ref{ThetaK} below.  
Finally, we point out that $\ACAo+\QFAC$ is also $\Pi_{2}^{1}$-conservative over $\ACA_{0}$ by \cite{yamayamaharehare}*{Theorem 2.2}.

\smallskip

%Finally, we prove the following theorem. 
To establish Theorem \ref{main1}, we now exhibit a model (aka type structure) of $\SIXK$ in which there is no special fan functional and in which $\HBU$ fails; hence $\SIXK$ cannot prove $\HBU$ by the soundness theorem. 
\begin{thm}\label{realmain1}
There is a type structure validating $\SIXK$ \(for all $k$\), and at the same time satisfying $(\forall \Theta^{3})\neg\SCF(\Theta)$ and $\neg\HBU$.  
\end{thm}
\begin{proof}
We introduce a family of type structures validating $(\forall \Theta^{3})\neg\SCF(\Theta)$.  Theorem~\ref{6.3} below tells us that one of those structures contains all $S_{k}^{1}$ and is closed under S1-S9, establishing the theorem. 
Intuitively speaking, we start from a $\beta$-model $A$ and have that any functional $G:A \rightarrow \N$ which is computable in some $S^2_k$ and elements from $A$ will be total over $\N^\N$ by the same algorithm.  
By absoluteness, there are $f_1 , \ldots , f_n$ in $A$ inducing a covering of $2^{\N}$ of the standard form. Since it is flexible which objects of type 2 we include in an extension of $A$ to a typed structure, $A$ together with the $S^2_k$'s cannot ``decide" whether there is $\Theta$ as in $\SCF(\Theta)$.

\smallskip

Let $A \subseteq \N^\N$ be a countable set such that all $\Pi^1_k$-statements with parameters from $A$ are absolute for $A$.
Also, let $S^2_k$ be the characteristic function of a complete $\Pi^1_k$-set for each $k$; we also write $S^2_k$ for the restriction of this functional to $A$. 
Clearly, for $f \in \N^\N$ computable in any $S^2_k$ and some $f_1 , \ldots , f_n$ from $A$, $f$ is also in $A$.
\begin{convention}\label{connie}\rm
Since $A$ is countable, we write $A$ as the increasing union $ \bigcup_{k \in \N}A_n$ where $A_0$ consists of the hyperarithmetical functions and for $k > 0$ we have:
\begin{itemize}
\item There is an element in $A_k$ enumerating $A_{k-1}$.
\item $A_k$ is the closure of a finite set $g_1 , \ldots , g_{n_k}$ under computability in $S^2_k$.
\end{itemize}
For the sake of uniform terminology, we rename $\exists^2$ to $S^2_0$ and let the associated finite sequence $g_1 , \ldots , g_{n_0}$ be the empty list. 
\end{convention}
We now define the functional $F^{2}$ on $A$ as follows.  
\bdefi[The functional $F$]\label{effkes}
Define $F(f)$ for $f \in A$ as follows:
\begin{itemize} 
\item If $f \not \in 2^\N$, put $F(f) := 0$.
\item If $f \in 2^{\N}$, let $k$ be minimal such that $f \in A_k$. We put $F(f): = 2^{-(k+2+e)}.$ where $e$ is a `minimal' index for computing $f$ from $S^2_k$ and $\{g_1 , \ldots , g_{n_k}\}$ as follows:
the ordinal rank of this computation of $f$ is minimal and $e$ is minimal among the indices for $f$ of the same ordinal rank.
\end{itemize}
\edefi
By definition, $F$ as in Definition \ref{effkes} is injective on $A_0$ and on each set $A_{k+1} \setminus A_k$. Moreover, if ${\bf m}$ is the usual measure on $2^{\N}$, we see that
\be\label{koriou}\textstyle
{\bf m}\big(\bigcup_{f \in A_0}[\bar fF(f)]\big) \leq 2^{-1} \textup{ and }{\bf m}\big(\bigcup_{f \in A_{k+1} \setminus A_k}[\bar fF(f)]\big) \leq 2^{-(k+2)}.
\ee
As a consequence, if $F$ is extended to a total functional $G$ and $\Theta$ satisfies $\SCF(\Theta)$, then $\Theta(G)$ cannot be a finite list from $A$. 
Similarly, a finite sequence $\langle f_{1}, \dots, f_{n} \rangle$ in $A$ is already in some $\cup_{k\leq m}A_{k}$, and \eqref{koriou} implies that $\cup_{i\leq n}[f_{i}\overline{G}(f_{i})]$ does not cover Cantor space, for any total extension $G$ of $F$.  

\smallskip

Thus, for any type structure ${\rm \textsf{\textup{Tp}}} = \{{\rm \textsf{\textup{Tp}}}_n\}_{n \in \N}$ where ${\rm \textsf{\textup{Tp}}}_0 = \N$, ${\rm \textsf{\textup{Tp}}}_1 = A$ and $F \in {\rm \textsf{\textup{Tp}}}_2$, there is no instance of $\Theta $ as in $\SCF(\Theta)$ in ${\rm \textsf{\textup{Tp}}}_3$.  To establish the theorem, we require one such type structure, containing each $S^2_k$ and $F$, and closed under Kleene's S1-S9; such a type structure is provided by Theorem \ref{6.3}, i.e.\ the latter establishes the theorem, and we are done.    
\end{proof}
For Theorem \ref{6.3}, we require some properties of $F^{2}$ from Definition \ref{effkes}.
\begin{lemma}[Properties of the functional $F$]\label{6.1} ~
\begin{enumerate}
\item  For each $k$, the restriction of $F$ to $A_k$ is computable in the functions $g_1 , \ldots , g_{n_k}$ from Convention \ref{connie}, and the functional $S^2_k$.  % within S1 - S9.
\item Let $G$ be any total extension of $F$, let $f_1 , \ldots , f_m \in A$, and assume that the function $f$ is computable in $G$, $f_1 , \ldots, f_m$ and some $S^2_k$. Then also $f \in A$. %\marginpar{\footnotesize{Corrected typo.}}
\end{enumerate}
\end{lemma}
\begin{proof}
For the first part, we use induction on $k$. For $k = 0$, we use \emph{Gandy selection} (\cite{longmann}*{p.\ 210}) for $\exists^2$ which permits us to compute an $\exists^2$ index for each hyperarithmetical function. 
For $k > 0$, we use that $S^2_l$ is computable in $S^2_k$ when $l < k$ 
and that we have enumerations of each of the sets $A_0, \ldots, A_{k-1}$ computable in $g_1 , \ldots , g_{n_k}$ and $S^2_k$. %\marginpar{\footnotesize{Added missing $S$.}}
Then we can apply the induction hypothesis for $f \in A_l$ for some $l < k$ and the Gandy selection method relative to $S^2_k$ for $f \in A_k \setminus A_{k-1}$.
For the second part, without loss of generality, we may assume that $f_1 , \ldots , f_m$ are all in $A_k$. By the first part of this lemma, $G$ restricted to $A_k$ is computable in $S^2_k$, and $A_k$ is closed under computations relative to $S^2_k$. The claim now follows.
\end{proof}
\begin{theorem}\label{6.3}
There is a type structure $\{{\rm \textsf{\textup{Tp}}}_n\}_{n \in \N}$, closed under Kleene's \textup{S1-S9}, such that ${\rm \textsf{\textup{Tp}}}_0 = \N$ and:
\begin{enumerate}
\item ${\rm \textsf{\textup{Tp}}}_1$ is a countable subset $A$ of $\N^\N$ such that all analytical statements \(i.e.\ any $\Pi^1_m$-sentence, for any $m$\) are absolute for $A$.
\item ${\rm \textsf{\textup{Tp}}}_2$ contains the restrictions of all $S^2_k$ to $A$.
\item There exists $F \in {\rm \textsf{\textup{Tp}}}_2$ inducing an open covering of $A$ for which there is no finite sub-covering in the type structure.
%\item ${\rm \textsf{\textup{Tp}}}_1$ is such that $\QFAC^{0, 1}$ holds.  
\end{enumerate}
\end{theorem}
\begin{proof}
The theorem expresses exactly what ${\rm \textsf{\textup{Tp}}}_n$ has to be for $n = 0$ and $n = 1$. For $n > 1$, we recursively let ${\rm \textsf{\textup{Tp}}}_n$ consist of all functionals $\phi:{\rm \textsf{\textup{Tp}}}_{n-1} \rightarrow \N$ that are S1-S9-computable in $F$, some $S^2_k$, and elements from $A$, where $F$ is as in Definition~\ref{effkes}. This type structure has the desired property.
Note that Feferman's $\mu$ is S1-S9-computable from $\exists^{2}$, and the former immediately yields $\QFAC^{1,0}$.  
\end{proof}
The proof of Theorem \ref{main1} is now done.  
As to the role of $\QFAC^{0,1}$, we show in \cite{dagsamV}*{\S4} that $\HBU$ is provable in $\Z_{2}^{\Omega}$, i.e.\ without $\QFAC^{0,1}$, as well as the construction of a type structure of $\SIXK+\QFAC^{0,1}$ in which $\neg\HBU$ holds. 
Thus, $\QFAC^{0,1}$ is not essential for obtaining $\HBU$.      
We finish this section with a remark. 
\begin{rem}[The Axiom of Choice and $\Theta$]\label{ThetaK}\rm
First of all, the (quantifier-free) Axiom of Choice is used to establish the existence of $\Theta$ in Theorem \ref{nolapdog}, while by \cite{samFLO}*{Cor.~3.29}, $\Theta$ can be \emph{computed} (via a term from G\"odel's $T$) from a version of $\exists^{3}$ enriched with quantifier-free choice.  However, Borel's construction from \cite{opborrelen}*{p.~52} can be applied to our notion of canonical cover, yielding a \emph{countable sub-cover} {without} using the Axiom of Choice.
Furthermore, the instance $\Theta_{0}$ of the special fan functional from \cite{dagsam}*{\S5.1} is defined using Borel's construction.  
\end{rem}

\subsection{Lindel\"of lemma and similar theorems}\label{korkske}
We establish results analogous to Theorem \ref{main1} for some of the other theorems from Remark \ref{bthm}.
We discuss how these theorems are used in mathematics in Remark~\ref{popo}.
\subsubsection{Lindel\"of lemma}
We recall that Lindel\"of proved the \emph{Lindel\"of lemma} in 1903 (\cite{blindeloef}), while Young and Riesz proved a similar theorem in 1902 and 1905 (\cite{manon, youngster}); this lemma expresses that any open cover of any subset of $\R^{n}$ has a countable sub-cover.    
We study variations of this lemma restricted to $\R$, while Baire space is studied in Section \ref{lindeb}.   
We believe $\LIND$ is the closest to Lindel\"of's original\footnote{Lindel\"of formulates his lemma in \cite{blindeloef}*{p.\ 698} as follows: \emph{Soit \textsf{\textup{(P)}} un ensemble quelconque situ\'e dans l'espace $\R^{n}$ et, de chaque point $\textsf{\textup{P}}$ comme centre, construisons une sph\`ere $\textsf{\textup{S}}_{\textsf{\textup{P}}}$ d'un rayon $\rho_{\textsf{\textup{P}}}$ qui peut varier de l'un point
\`a l'autre; il existe une infinit\'e d\'enombrable de ces sph\`eres de telle sorte que tout point de l'ensemble donn\'e
soit int\'erieur \`a au moins l'une d'elles}.  Applying $\QFAC^{0,1}$ to $\LIND$, one could obtain $\Phi^{0\di1}$ such that $(\forall n \in\N)[(a_{n}, b_{n}) = I_{\Phi(n)}^{\Psi} ]$, but such a functional is nowhere to be found in Lindel\"of's original formulation.}.  
%We are interested in the following formulations, of which $\LIND_{2}$ seems to be closest to Lindel\"of's original\footnote{Lindel\"of formulates the conclusion of his lemma in \cite{blindeloef}*{p.\ 698} as \emph{il existe une infinit\'e
%d\'enombrable de ces sph\`eres}; in the context of set theory (which is mentioned in the title of \cite{blindeloef}) this quote can (only?) reasonably be interpreted as \emph{there exists a countable collection of such balls}.  This `countable collection' is given by the function $\Phi$ in $\LIND_{2}$.}  formulation.  % from \cite{blindeloef}*{p.\ 698}.  
\bdefi[$\LIND$] 
For every $\Psi:\R\di \R^{+}$, there is a sequence of open intervals $\cup_{n\in \N}(a_{n}, b_{n})$ covering $\R$ such that $(\forall n \in\N)(\exists x \in \R)[(a_{n}, b_{n}) = I_{x}^{\Psi} ]$.  % and $\cup_{n}(a_{n},b_{n})$ covers $\R$.
\edefi
\bdefi[$\LIND_{2}$] 
$(\forall \Psi:\R\di \R^{+})(\exists \Phi^{0\di 1})(\forall x\in \R)(\exists n^{0})(x\in I^{\Psi}_{\Phi(n)})$.
%For every $\Psi^{2}$, there is a sequence of open intervals $(a_{n}, b_{n})$ such that $(\forall n^{0})(\exists x^{1}\in I)[(a_{n}, b_{n}) = I_{x}^{\Psi} ]$ and $[0,1]\subset \cup_{n}(a_{n},b_{n})$.
\edefi
\bdefi[$\LIND_{3}$] 
$(\exists \Xi)(\forall \Psi:\R\di \R^{+})(\forall x\in \R)(\exists n^{0})(x\in I^{\Psi}_{\Xi(\Psi)(n)})$.
%For every $\Psi^{2}$, there is a sequence of open intervals $(a_{n}, b_{n})$ such that $(\forall n^{0})(\exists x^{1}\in I)[(a_{n}, b_{n}) = I_{x}^{\Psi} ]$ and $[0,1]\subset \cup_{n}(a_{n},b_{n})$.
\edefi
The following theorem establishes the connection between $\LIND$ and $\HBU$, while also showing that 
the introduction of $\Xi$ or $\Phi$ does not change $\LIND$ much.  
\begin{thm}\label{crucks}
The system $\RCAo+\QFAC^{0,1}$ proves $[\LIND+\WKL]\asa \HBU$ and $\ACAo+\QFAC^{0,1}$ proves $\LIND\asa \LIND_{2}\asa \LIND_{3}$.
\end{thm}
\begin{proof}
%NEW
For the first part, $\WKL_{0}$ implies that every \emph{countable} cover of $I$ has a finite sub-cover by \cite{simpson2}*{IV.1.2}.  Hence, $\LIND+\WKL_{0}\di \HBU\di \WKL_{0}$ is immediate, while $\HBU$ clearly generalises to $[-N, N]$ for any natural number $N^{0}$.  Putting all the finite sub-covers of $[-N, N]$ together (using $\mu^{2}$ and $\QFAC^{0,1}$), one obtains the countable cover needed for $\LIND$, \emph{assuming} $(\mu^{2})$.
On the other hand, if $\neg(\mu^{2})$ then all functions on the reals are continuous by \cite{kohlenbach2}*{Prop.\ 3.9 and 3.12}.  But $\cup_{q\in \Q}I_{q}^{\Psi}$ is a countable sub-cover of the canonical sub-cover for \emph{continuous} $\Psi:\R\di \R^{+}$, and hence $\LIND$ follows.  The law of excluded middle $(\mu^{2})\vee \neg (\mu^{2})$ now finishes this part.    

\smallskip

For the second part, we only need to prove the forward implications.  So assume $\LIND$ and note that the formula `$(a_{n}, b_{n}) = I_{x}^{\Psi} $' is just $a_{n}=_{\R}x-\Psi(x)\wedge b_{n}=_{\R}x+\Psi(x)$, which is $\Pi_{1}^{0}$, i.e.\ this formula is decidable using $\mu^{2}$, and we can treat it as quantifier-free in $\ACAo$.  Now apply $\QFAC^{0,1}$ to $(\forall n \in\N)(\exists x \in \R)[(a_{n}, b_{n}) = I_{x}^{\Psi} ]$ to obtain $\LIND_{2}$.  For the final implication, we use the same argument as in the first part, establishing $\HBU$ relativised to $[-N,N]$, and now combined with the existence of the functional $\Omega$ as in \eqref{1337}.  
%NEW&&& do we need \mu?
\end{proof}
The local equivalence of `epsilon-delta' and sequential continuity is not provable in \textsf{ZF}, while $\QFAC^{0,1}$ suffices to establish the equivalence in a general context (See \cite{kohlenbach2}*{Rem.\ 3.13} for details).  
It is then a natural question whether the use of $\QFAC^{0,1}$ in the theorem is similarly essential.  
This question is deviously subtle, as discussed in Remark \ref{linpinpon}.
% one expects that $\QFAC^{0,1}$ is essential, which is part of the results 
% but $\Z_{2}^{\Omega}$ proves $\HBU$ (See \cite{dagsamV}*{\S4}), $\textsf{ZF}$ does not prove $\HBU\asa [\LIN+\WKL]$, i.e.\ $\QFAC^{0,1}$ is essential.  
%We also believe that the axiom $\QFAC^{0,1}$ is essential in proving the second part.  
We show in \cite{dagsamV}*{\S4} that $\ACAo$ in Theorems \ref{nolapdog} and~\ref{crucks} can be weakened to $\RCAo$ plus the existence of the \emph{classical} fan functional.  % as shown in \cite{dagsamV}*{\S4}.  
%
%\smallskip
%
%We now have the following theorem. 
\begin{thm}\label{main2}
 $\Z_{2}^{\Omega}+\QFAC^{0,1}$ proves $\LIND$; no system $\SIXK$ proves $\LIND$.  
\end{thm}
\begin{proof}
The second part is immediate from Theorems \ref{main1} and \ref{crucks}.  The first part is proved by proving item \eqref{NFP} from Remark \ref{bthm} in $\Z_{2}^{\Omega}$, and deriving $\LIND$ from this item.  
Thus, consider the following for any $\Pi_{\infty}^{1}$-formula $A$ with any parameter:
\be\tag{$\NFP$}
(\forall f^{1})(\exists n^{0})A(\overline{f}n)\di (\exists \gamma^{1}\in K_{0})(\forall f^{1})A(\overline{f}\gamma(f)).
\ee
Here, `$\gamma^{1}\in K_{0}$' expresses that $\gamma^{1}$ is an \emph{associate}, which is the same as a \emph{code} from RM by \cite{kohlenbach4}*{Prop.\ 4.4}.  Formally, `$\gamma^{1}\in K_{0}$' is the following formula:
\[
(\forall f^{1})(\exists n^{0})(\gamma(\overline{f}n)>_{0}0) \wedge (\forall n^{0}, m^{0},f^{1}, )(m>n \wedge \gamma(\overline{f}n)>0\di  \gamma(\overline{f}n)=_{0} \gamma(\overline{f}m) ).
\]
The value $\gamma(f)$ for $\gamma\in K_{0}$ is defined as the unique $\gamma(\overline{f}n)-1$ for $n$ large enough.  
Now, since $A$ as in $\NFP$ is a $\Pi_{k}^{1}$-formula for some $k$, we may treat it as quantifier-free given $(\exists^{3})$.  
Applying $\QFAC^{1,0}$ to the antecedent of $\NFP$, there is $Y^{2}$ such that $(\forall f^{1})A(\overline{f}Y(f))$.  Define $Z^{2}$ using $(\exists^{3})$ as follows: $Z(f)$ is the least $n\leq Y(f)$ such that $A(\overline{f}n)$ if it exists, and zero otherwise.  
Note that $Z$ is continuous on $\N^{\N}$ and hence has an associate by \cite{kohlenbach4}*{Prop.\ 4.7}.  Alternatively, define the associate $\gamma^{1}$ directly as follows:  for $w^{0^{*}}$, 
define $\gamma(w)$ as the least $n\leq |w|$ such that $A(\overline{w}n)$ if such there is, and zero otherwise.  
Clearly, we have $\gamma\in K_{0}$ and $(\forall f^{1})A(\overline{f}\gamma(f))$, i.e.\ $\NFP$ follows.  Finally, $\LIND$ follows from the latter by considering: 
\be\label{xxx}\textstyle
(\forall x\in \R)(\exists n\in \N)\big[(\exists y\in \R)(  ([x](\frac{1}{2^{n}})-\frac{1}{n}, [x](\frac{1}{2^{n}})+\frac{1}{n})\subset I_{y}^{\Psi}  )\big]
\ee
for $\Psi:\R\di \R^{+}$, and where the formula in square brackets is abbreviated $A(\overline{x}n)$.  This is a slight abuse of notation, as (only) the first $2^{n}$ elements in the sequence $x^{1}$ are being used in \eqref{xxx}.
Applying $\NFP$ to \eqref{xxx}, we obtain $\gamma\in K_{0}$ such that:  % witnessing the number quantifier in \eqref{xxx}.  
\be\label{xxx2}\textstyle
(\forall x\in \R)(\exists y\in \R)\big[  ([x](\frac{1}{2^{\gamma(x)}})-\frac{1}{\gamma{(x)}}, [x](\frac{1}{2^{\gamma{(x)}}})+\frac{1}{\gamma{(x)}})\subset I_{y}^{\Psi}  \big].
\ee
Note that the formula in square brackets in \eqref{xxx2} is arithmetical (including the formula needed to make the notation $\gamma(x)$ work).
Hence, using $\QFAC^{0,1}$ and $(\mu^{2})$, there is a functional $\Phi$ which provides the real $y$ from \eqref{xxx2} on input $x\in \Q$.
The countable sub-cover of $\cup_{x\in \R}I_{x}^{\Psi}$ can then be found by enumerating $\Phi(q_{w})$ for
all finite sequences $w^{0^{*}}$ of {rationals} which represent rationals $q_{w}^{0}$ and are such that $\gamma(w)>_{0}0$.  In particular, every $x\in \R$ is in some $I_{y}^{\Psi}$ by \eqref{xxx2}, and since $v^{0^{*}}:=\overline{x}2^{\gamma(x)}$ is in the aforementioned enumeration, we also have $x\in I_{\Phi(q_{v})}^{\Psi}$.
% and $\LIND$ readily follows.     
\end{proof}
By the first part of Theorem \ref{crucks}, the results regarding $\LIND$ have to be somewhat similar to those for $\HBU$.  However, the Lindel\"of theorem \emph{for the Baire space} behaves quite differently, as will be established in Section \ref{RMR2}.  
Furthermore, while $\HBU$ implies $\WKL$, $\LIND$ does not by the following corollary.
\begin{cor}\label{dirfi}
The system $\RCAo+\LIND$ proves the same $\L_{2}$-sentences as $\RCA_{0}$.  
\end{cor}  
\begin{proof}
By the proof of \cite{kohlenbach2}*{Prop.\ 3.1}, if for a sentence $A\in \L_{\omega}$, the system $\RCAo$ proves $A$, then $\RCA_{0}$ proves $[A]_{\ECF}$, 
where `$[~\cdot~]_{\ECF}$' is a syntactic translation which -intuitively- replaces any object of type $2$ or higher by a code $\gamma^{1}\in K_{0}$.  
Thus, to establish the corollary, it suffices to show that $[\LIND]_{\ECF}$ is provable in $\RCA_{0}$.  
However, $\LIND$ only involves objects of type $0$ and $1$, except for the leading quantifier.  Hence, $[\LIND]_{\ECF}$ is nothing more than $\LIND$ with `$(\forall \Psi^{1\di 1})$' replaced by `$(\forall \gamma^{1}\in K_{0})$'.
Thus, by enumerating $\gamma(w)$ as in the proof of the theorem, we immediately obtain a countable sub-cover, and $[\LIND]_{\ECF}$ is provable in $\RCA_{0}$.    
\end{proof}
Finally, we discuss the Lindel\"of lemma in the grand scheme of things, and associated results to be proved in \cite{dagsamV}.  
\begin{rem}\label{linpinpon} \rm 
The following are equivalent over $\textsf{ZF}$ by \cite{heerlijk}: (i) $\R$ is a Lindel\"of space, and (ii) the axiom of countable choice (for subsets of $\R$).  
This resonates with the use of $\QFAC^{0,1}$ in Theorem \ref{main2}, but is not the entire story: we introduce a weak and a strong version of $\LIND$ (and $\HBU$) in \cite{dagsamV} based on the 1895 and 1899 proofs of the Heine-Borel theorem by Borel (\cite{opborrelen}) and Schoenflies (\cite{schoen2}).  The \emph{weak} version of $\LIND$ (and $\HBU$) is provable in $\Z_{2}^{\Omega}$ (and hence in $\ZF$).  This is possible as the weak version only provides the \emph{existence} of a countable sub-cover as in $\LIND$, while the strong version \emph{additionally} identifies a sequence of reals which yield the countable sub-cover, as in $\LIND_{2}$ via $\Phi^{0\di 1}$. 
\end{rem}

\subsubsection{Other theorems}\label{other}
We discuss how the theorems in Remark \ref{bthm} imply either $\LIND$ or $\HBU$, and hence have similar properties to the latter.    
\begin{enumerate}
\item The \emph{Besicovitsch and Vitali\footnote{Not to be confused with the Vitali covering \emph{theorem} (\cite{bartle}*{p.\ 79}), which does follow from the Vitali covering \emph{lemma} via Banach's proof from \cite{naatjenaaien}*{p.\ 81}; we believe that the Vitali covering theorem is weaker than $\HBU$, but nonetheless requires full second-order arithmetic $\Z_{2}^{\Omega}$ to prove. } covering lemmas} as in \cite{auke}*{\S2} start from a cover of open balls, one for each $x\in E\subset \R^{n}$, and states the existence 
of a countable sub-cover of $E$ with nice properties, i.e.\ $\LIND$ follows.  Note that Vitali already (explicitly) discussed uncountable covers in \cite{vitaliorg}*{p.\ 236}.  
\item The existence of \emph{Lebesgue numbers} for {any} open cover is equivalent to $\HBU$, in the same way the countable case is equivalent to $\WKL_{0}$ (\cite{moregusto}*{Theorem~5.5}). 
The same holds for the \emph{Banach-Alaoglu theorem}; the equivalence between the countable case and $\WKL_{0}$ is established in \cite{xbrownphd}*{p.\ 140}.
%\item The \smallskip
\item The principle $\NFP$ implies $\LIND$ by the proof of Theorem \ref{main2}.  
\item The \emph{Heine-Young} and \emph{Lusin-Young} theorems from \cite{YY} are clearly refinements of $\HBU$, while the  \emph{tile theorem} \cite{YY, wildehilde}, and the latter's generalisation due to Rademacher (\cite{rademachen}*{p.\ 190}) are clearly refinements of $\LIND$. 
\item Basic properties of the \emph{gauge integral}, like uniqueness and its extension of the Riemann integral, are equivalent to $\HBU$
over the system $\ACAo$, as shown in Section \ref{introgau}.  Note that $\ACAo$ is very weak compared to $\Z_{2}^{\Omega}$, which is in turn required to prove $\HBU$ by Theorem \ref{main1}.   
\item In Section \ref{klipel}, we derive $\HBU$ from the following generalisation of the Bolzano-Weierstrass theorem: \emph{every net in the unit interval has a convergent sub-net}. 
Nets (aka Moore-Smith sequences) provide a generalisation of the concept of sequences beyond countable index sets.     
\end{enumerate}
%We also point out that the notion of \emph{topological entropy}, introduced in \cite{adelaar}, crucially hinges on the existence of finite sub-cover.  
%
%\smallskip
%
Finally, we discuss how some of the `countable covering theorems', like the Lindel\"of and Vitali lemmas, from Remark~\ref{bthm} are used in mathematics.  
%mathematical and foundational reasons why the Cousin lemma is different from the other covering theorems in Remark \ref{bthm}.
\begin{rem}\label{popo}\rm
The Cousin lemma is special because it deals with bounded sets (essentially the unit interval), while the other covering theorems apply to unbounded sets (e.g.\ $\R^{n}$).  
Now, a cover of the latter is generally difficult to handle, but any \emph{countable} sub-cover `automatically' has nice properties: e.g.\ the \emph{countable} sub-additivity of the Lebesgue measure. 
In fact, the proofs of \emph{Sard's theorem} and the \emph{maximal theorem} in \cite{auke}, and of the \emph{Lebesgue density theorem} in \cite{grotesier} are based on this idea.  
%In other words, the non-local character of some of the covering theorems in Remark 1.1 is important for some real applications of these theorems in mathematics.
In other words, the non-local character of some of the covering theorems in Remark \ref{bthm} is important for some real applications of these theorems.  

\smallskip

Similarly, for properties which hold in the unit interval \emph{minus a measure zero set}, like the \emph{differentiation theorem} for gauge integrals (\cite{bartle}*{p.\ 80}), one uses the Vitali covering theorem to provide a countable sub-cover in which the complement of a finite sub-sub-cover has small length.  Hence, one can neglect this complement and the finite nature of the sub-sub-cover then makes the proof straightforward.  
\end{rem}
%\section{Mathematical significance}
\subsection{The gauge integral}\label{introgau}
\subsubsection{Introduction}
We provide a brief introduction to the \emph{gauge integral} (Section~\ref{bintro}) and establish that basic properties of this integral, like uniqueness and the fact it extends the Riemann and Lebesgue integral, are equivalent to $\HBU$ (Section \ref{grm}) over the (relatively weak) system $\ACAo$.  
The gauge integral is not an isolated incident: the \emph{Henstock variational measure} is a generalisation of the Lebesgue outer measure (see \cite{leekessjorsj}*{Ch.\ 5}), and its basic properties are equivalent to $\HBU$ as well, as will be shown in future work. 
We address a possible criticism that may be levelled at our results in Section \ref{extradata}.  For the rest of this section, we motivate the RM-study of the gauge integral.    

\smallskip

As will become clear below, the gauge integral enjoys the conceptual simplicity of the Riemann integral, but also has greater generality than the Lebesgue integral.  
In fact, the gauge integral boasts the most general version of the fundamental theorem of calculus (see Theorem \ref{FTC}) and is `maximally' closed under improper integrals (Hake's theorem; see Theorem \ref{coromag} for a special case).  
For these reasons, there have been calls for (a somewhat stripped-down version of) the gauge integral to replace the Riemann and Lebesgue integral (and the associated measure theory) in the undergraduate curriculum (\cite{bartleol, bartleol2,bartleol3}).  In a nutshell, the gauge integral can only be called \emph{natural and mainstream} from the point of view of mathematics.  

\smallskip

Regarding the connection between physics and the gauge integral, Muldowney has expressed the following opinion in a private communication.  
\begin{quote}
There are a number of different approaches to the formalisation of Feynman’s path integral.  However, 
if one requires the formalisation to be close to Feynman’s original formulation, then the gauge integral
is really the only approach.
\end{quote}
Arguments for this opinion, including major contributions to Rota's program for the Feyman integral, may be found in \cite{mully}*{\S A.2}.
Another argument in favour of the gauge integral is that this formalism gives rise to so-called {physical} solutions, i.e.\ in line with the observations from physics  (see \cite{pouly,nopouly,nopouly2,nopouly3}). 
For instance, most path integral formalisms somehow require the concept of \emph{imaginary} time, while the gauge integral provides a more natural framework \emph{based on real time}. 
A major problem with imaginary time is namely the lack of an arrow/direction of time (\cite{goodtime}), where the latter is proscribed by thermodynamics (see also \cite{geiledel}).

\smallskip

Finally, another argument for the RM study of the gauge integral is as follows: the study of the equivalent\footnote{The gauge integral is equivalent to the (much older) Denjoy integral; see \cites{gordonkordon, bartle} for details.} Denjoy integral in descriptive set theory in \cite{walshout}*{\S2} makes essential use of fundamental results from \cite{zwette, gordonkordon}, like Hake's theorem for the gauge integral (see Theorem \ref{coromag}).  We note that the treatment in \cite{walshout} assumes the measurability of the integrand in the fundamental theorem of calculus, a condition that flies in the face of the generality the gauge integral enjoys (see Theorem \ref{FTC}).  

\subsubsection{Introducing the gauge integral}\label{bintro}
The gauge integral is a generalisation of the Lebesgue and (improper) Riemann integral; it was introduced by Denjoy (in a different from) around 1912 and studied by Lusin, Perron, Henstock, and Kurzweil.  
The exact definition is in Definition \ref{GI}, which we intuitively motivate as follows.

\smallskip

A limitation of the `$\eps$-$\delta$-definition' of the Riemann integral is that near a singularity of a function $f:[0,1]\di \R$, changes smaller than any fixed $\delta>0$ in $x$ can still result in huge changes in $f(x)$, guaranteeing that the associated Riemann sums vary (much) more than the given $\eps>0$.  The gauge integral solves this problem by replacing the \emph{fixed} $\delta>0$ with a \emph{gauge function} $\delta:\R\di \R^{+}$;  the latter can single out those partitions with `many' partition points near singularities to compensate for the extreme behaviour there.  
Similarly, $\delta:\R\di \R^{+}$ can single out partitions which avoid `small' sets whose contribution to the Riemann sums should be negligible. 
 We study $\frac{1}{\sqrt{x}}$ and \emph{Dirichlet's function} in Example \ref{fore} after the following definition.
\bdefi\label{GI}[Gauge integral]\label{GID}~   
\begin{enumerate}
\renewcommand{\theenumi}{\roman{enumi}}
\item A gauge on $I\equiv[0,1]$ is any function $\delta:\R\di \R^{+}$.  
\item A sequence $P:=(t_{0}, I_{0}, \dots, t_{k}, I_{k})$ is a \emph{tagged partition} of $I$, written `$P\in \textsf{tp}$', if the `tag' $t_{i}\in \R$ is in the interval $ I_{i}$ for $i\leq k$, and the $I_{i}$ partition $I$.
\item If $\delta$ is a gauge on $I$ and $P=(t_{i}, I_{i})_{i\leq k}$ is a tagged partition of $I$, then $P$ is \emph{$\delta$-fine} if $I_{i}\subseteq [t_{i}-\delta(t_{i}), t_{i}+\delta(t_{i})]	 $ for $i\leq k$.\label{hurfo}
\item For a tagged partition $P=(t_{i}, I_{i})_{i\leq k}$ of $I$ and any $f$, the \emph{Riemann sum} $S(f, P)$ is $\sum_{i=0}^{n}f(t_{i})|I_{i}|$, while the \emph{mesh} $\|P\|$ is $\max_{i\leq n}|I_{i}|$.
\item A function $f:I\di \R$ is \emph{Riemann integrable} on $I$ if there is $A\in \R$ such that $(\forall \eps>_{\R}0)(\exists\delta>_{\R}0)(\forall P\in \textsf{tp})(\|P\|\leq_{\R} \delta\di |S(f, P)-A|<_{\R}\eps)$.\label{lakel}
\item A function $f:I\di \R$ is \emph{gauge integrable} on $I$ if there is $A\in \R$ such that
$(\forall \eps>_{\R}0)(\exists\delta:\R\di \R^{+})(\forall P\in \textsf{tp})(\textup{$P$ is $\delta$-fine }\di |S(f, P)-A|<_{\R}\eps)$.\label{lakel2}
\item A \emph{gauge modulus} for $f$ is a function $\Phi:\R\di (\R\di \R^{+})$ such that $\Phi(\eps)$ is a gauge as in the previous item for all $\eps>_{\R}0$.  \label{foralleke}
\end{enumerate}
\edefi
The real $A$ from items \eqref{lakel} and \eqref{lakel2} in Definition \ref{GID} is resp.\ called the Riemann and gauge integral. 
We will always interpret $\int_{a}^{b}f $ as a gauge integral, unless explicitly stated otherwise.  We abbreviate `Riemann integration' to `R-integration', and the same for related notions.  
The following examples are well-known. 
\begin{exa}[Two examples]\label{fore}\rm
Let $f$ be the function ${1}/{\sqrt{x}}$ for $x>0$, and zero otherwise.  It is easy to show $\int_{0}^{1}f =_{\R}2$ using the gauge modulus $\delta_{\eps}(x):=\eps x^{2}$ for $x>0$ and $\eps^{2}$ otherwise.  
Let $g$ be constant $1$ for $x\in \Q$, and zero otherwise.  It is easy to show $\int_{0}^{1}g =_{\R}0$ using the gauge modulus $\delta_{\eps}(x):= 1$ if $x\not \in \Q$ and $\eps/2^{k+1}$ if $x$ equals the $k$-th rational (for some enumeration of the rationals fixed in advance).  
\end{exa}
Now, using the Axiom of Choice, a gauge integrable function always has a gauge modulus, but this is not the case in weak systems like $\RCAo$.
However, to establish the \emph{Cauchy criterion} for gauge integrals as in Theorem \ref{cauct}, a gauge modulus is essential.  
For this reason, we sometimes assume a gauge modulus when studying the RM of the gauge integral in Section \ref{grm}.  
Similar `constructive enrichments' or `extra data' exist in Friedman-Simpson RM, as established by Kohlenbach in \cite{kohlenbach4}*{\S4}.
Finally, the `standard' proof of the fundamental theorem of calculus for the gauge integral (see \cites{bartle,mullingitover,zwette, leekessjorsj} and Theorem \ref{FTC}) readily establishes 
the existence of a gauge modulus in terms of other\footnote{The proofs in \cites{bartle,mullingitover,zwette,leekessjorsj} establish that a gauge modulus for the integral in the fundamental theorem of calculus $\int_{a}^{b}F' = F(b)-F(a)$ is simply any modulus of differentiability of $F$.} extra data.   
%Nonetheless, 
% obtain a gauge modulus for gauge integrable functions on the unit interval in $\RCAo+\WKL+(\kappa_{0}^{3})+\QFAC^{0,2}$.

\subsubsection{Reverse Mathematics of the gauge integral}\label{grm}
We show that basic properties of the gauge integral are equivalent to $\HBU$.  
We have based this development on Bartle's introductory monograph \cite{bartle}. 

\smallskip

First of all, we show that $\HBU$ is equivalent to the uniqueness of the gauge integral, and to the fact that the latter extends the R-integral.
Note that the names of the two items in the theorem are from \cite{bartle}*{p.\ 13-14}.  
\begin{thm}\label{firstje} 
Over $\ACA_{0}^{\omega}$, the following are equivalent to $\HBU$:
\begin{enumerate}
\renewcommand{\theenumi}{\roman{enumi}}
\item Uniqueness: If a function is gauge integrable on $[0,1]$, then the gauge integral is unique. \label{itemone}
\item Consistency: If a function is R-integrable on $[0,1]$, then it is gauge integrable there, and the two
integrals are equal.\label{itemtwo}
\end{enumerate}
\end{thm}
\begin{proof}
We prove $\HBU\di \eqref{itemone}\di \eqref{itemtwo}\di \HBU$, where only the first implication requires $(\mu^{2})$.
To prove that $\HBU$ implies \emph{Uniqueness}, we assume the former and first prove that for every $\delta:\R\di \R^{+}$ there \emph{exists} a $\delta$-fine tagged partition.
To this end, apply $\HBU$ to $\cup_{x\in I}(x-\delta(x), x+\delta(x))$ to obtain a finite sub-cover $w:=(y_{0}, \dots, y_{k})$, i.e.\ we have $I\subset \cup_{x\in w}(x-\delta(x), x+\delta(x))$.   The latter cover is readily converted into a tagged partition $P_{0}:=(z_{j}, I_{j})_{j\leq l}$ (with $l\leq k$ and $z_{j}\in w$ for $j\leq l$) by removing overlapping segments and omitting redundant intervals `from left to right'.  By definition, $z_{j}\in I_{j}\subset (z_{j}-\delta(z_{j}), z_{j}+\delta(z_{j}))$ for $j\leq l$, i.e.\ $P_{0}$ is $\delta$-fine.  
%While the previous two steps are straightforward, it should be noted that (i) $\HBU$ is essential by the equivalences in the theorem, and (ii) to convert $w$ into a tagged partition, we need to compare real numbers (in the sense of deciding whether $x>_{\R}0$ or not) and this operation is only available in $\ACAo$.

\smallskip
\noindent
Now let $f$ be gauge integrable on $I$ and suppose we have for $i=1,2$ ($A_{i}\in \R$) that:
\be\label{loker}
(\forall \eps>0)(\exists\delta_{i}^{1}:\R\di \R^{+})(\forall P\in \textsf{tp})(\textup{{$P$ is $\delta_{i}$-fine} }\di |S(f, P)-A_{i}|<\eps).
\ee
Fix $\eps>0$ and $\delta_{i}:\R\di \R^{+}$ in \eqref{loker} for $i=1,2$.  We define $\delta_{3}:\R\di \R^{+}$ as $\delta_{3}(x):=\min (\delta_{1}(x), \delta_{2}(x))$.  
By definition, a partition which is $\delta_{3}$-fine, is also $\delta_{i}$-finite for $i=1,2$.  
Now let the partition $P_{0}\in \textsf{tp}$ be $\delta_{3}$-fine, and derive the following:
\[
|A_{1}-A_{2}|=_{\R}|A_{1}-S(f, P_{0})+S(f, P_{0})-A_{2}|\leq_{\R}|A_{1}-S(f, P_{0})|+|S(f, P_{0})-A_{2}|\leq_{\R} 2\eps.
\]
Hence, we must have $A_{1}=_{\R}A_{2}$, and \emph{Uniqueness} follows. 

\smallskip

To prove that \emph{Uniqueness} implies \emph{Consistency}, note that `$P$ is $d_{\delta}$-fine' is equivalent to `$\|P\|\leq \delta$' for the gauge $d_{\delta}:\R\di \R^{+}$ which is constant $\delta>0$.
Rewriting the definition of Riemann integration with this equivalence, we observe that an R-integrable function $f$ is also gauge integrable (with a constant gauge $d_{\delta}$ for every choice of $\eps>0$).    
The assumption \emph{Uniqueness} then guarantees that the number $A$ is the only possible gauge integral for $f$ on $I$, i.e.\ the two integrals are equal.      

\smallskip

To prove that \emph{Consistency} implies $\HBU$, suppose the latter is false, i.e.\ there is $\Psi_{0}:\R\di \R^{+}$ such that $\cup_{x\in I}I_{x}^{\Psi_{0}}$ does not have a finite sub-cover. 
Now let $f:I\di \R$ be R-integrable with R-integral $A\in \R$.
Define the gauge $\delta_{0}$ as $\delta_{0}(x):={\Psi_{0}(x)}$ and note that for any $P\in \textsf{tp}$, we have that $P$ is \emph{not} $\delta_{0}$-fine, as $\cup_{x\in I}I_{x}^{\Psi_{0}}$ would otherwise have a finite sub-cover (provided by the tags of $P$).  Hence, the following statement is vacuously true, as the underlined part is false:
\be\label{corefu}
(\forall \eps>0)(\forall P\in \textsf{tp})(\underline{\textup{$P$ is $\delta_{0}$-fine }}\di |S(f, P)-(A+1)|<\eps).
\ee
However, \eqref{corefu} implies that $f$ is gauge integrable with gauge $\delta_{0}$ and gauge integral $A+1$, i.e.\ \emph{Consistency} is false as the Riemann and gauge integrals of $f$ differ.  Note that $\delta_{0}$ also provides a gauge \emph{modulus} by \eqref{corefu} in case $\neg\HBU$.  
\end{proof}
By the above, the role of $\HBU$ in making the gauge integral well-behaved, consists in avoiding \eqref{loker} and \eqref{corefu} being vacuously true due to the absence of $\delta_{i}$-fine partitions (for $i=0,1,2$).  
Thus, the Cousin lemma is called \emph{Fineness theorem} in \cite{bartle}.
As will become clear below, this is the \emph{sole} role of $\HBU$ in this context.  Nonetheless, $\HBU$ features in the RM of topology and uniform theorems in \cite{dagsamV, sahotop}, and Remark \ref{engauged} suggest an important role for the special fan functional, a realiser for $\HBU$, in gauge integration.
%We emphasise the crucial nature of this existence: \eqref{loker} is vacuously true if there is no $\delta_{i}$-fine tagged partition; in other words: we can \emph{only} make meaningful use of the conclusion 
%of \eqref{loker}, \emph{if} we show the existence of a $\delta_{i}$-fine tagged partition.  

\smallskip

In passing, we discuss the question if $\ACAo$ in the previous (and subsequent) theorem can be weakened to $\RCAo$.  
In our opinion, this weakening would not be spectacular, given that $\HBU$ requires $\Z_{2}^{\Omega}$ for a proof, as established above. 
Furthermore, even very basic properties of the gauge integral require $\ACAo$, as follows.  
\begin{exa}[Splitting the domain]\label{splitskop}\rm
As it turns out,  
proving $\int_{0}^{1}f=_{\R}\int_{0}^{x}f+\int_{x}^{1}f$ for $0<_{\R}x<_{\R}1$ \emph{in general} seems to require a \emph{discontinuous} gauge.    
Indeed, if for $\eps>0$ the functions $\delta_{1}, \delta_{2}$ are gauges for the right-hand side of the equation, a gauge for the left-hand side is as follows (\cite{bartle}*{p.\ 45}):%
\be\label{hersing}
\delta_{3}(y):=
\begin{cases}
\min(\delta_{1}(y), \frac{1}{2}(x-y)) &  y\in [0, x)\\
\min(\delta_{1}(x), \delta_{2}(x)) &  y=_{\R}x\\
\min(\delta_{2}(y), \frac{1}{2}(y-x)) &  y\in (x, 1]\\
\end{cases}
\ee
The function $\delta_{3}$ is discontinuous in general, but can be defined in $\ACAo$. 
\end{exa}
Secondly, we prove the \emph{Cauchy criterion} for gauge integrals, as this theorem is needed below.  Our proof is based on \cite{bartle}*{p.\ 40} and illuminates the role of $\Theta$.
\begin{thm}[$\ACAo+\HBU+\QFAC^{2,1}$; Cauchy criterion]\label{cauct} A function $f:I\di \R$ is gauge integrable with a modulus if and only if there is $\Phi:\R^{+}\di (\R\di \R^{+})$ with 
\be\label{tilda}
(\forall \eps>_{\R}0)(\forall P, Q\in \textsf{\textup{tp}})(\textup{$P, Q$ are $\Phi(\eps)$-fine }\di |S(f, P)-S(f, Q)|<_{\R}\eps).
\ee
\end{thm}
\begin{proof}
The forward implication follows by considering a gauge modulus $\Phi$ for $f$ and 
\[
|S(f, P)-S(f, Q)|=|S(f, P)-A+A- S(f, Q)|\leq |S(f, P)-A|+|A- S(f, Q)| \leq \eps
\]
where $P, Q$ are $\Phi(\eps/2)$-fine and $A$ is the gauge integral of $f$ over $I$.  For the reverse implication let $\Phi$ be as in \eqref{tilda}; we need to find the real $A$ from the definition of gauge integration.  
This real $A$ can be obtained as the limit of the sequence $S(f, Q_{n})$ where $Q_{n}$ is a $\Phi(\frac{1}{2^{n}})$-fine partition.  
Now, these partitions $Q_{n}$ can in turn be defined by applying the functional $\Omega$ from Theorem \ref{nolapdog} to the canonical cover associated to $\Phi(\frac{1}{2^{n}})$
and using Feferman's $\mu$ to convert the resulting finite sub-cover to a suitable partition.  Finally, \eqref{tilda} guarantees that the sequence $S(f, Q_{n})$ is Cauchy, while $\ACA_{0}$ proves that a Cauchy sequence has a limit by \cite{simpson2}*{III.2.2}.  
\end{proof}
The previous proof explains the need for a gauge modulus: the latter is essential in `reconstructing' the gauge integral $A$ as the limit in the proof, if $A$ is not given.  

\smallskip

Thirdly, we show that $\HBU$ is equivalent to the fact that the gauge integral encompasses the \emph{improper} R-integral.  
The latter is a (usual) R-integral $\int_{a}^{b}f(x)d(x)$ where additionally a limit operation like $\lim_{a\di 0}$ or $\lim_{b\di \infty}$ is applied.  
This method allows one to consider unbounded domains or use singularities as end points; as suggested by its name, an \emph{improper} R-integral 
is (generally) not an actual R-integral.  Now, \emph{Hake's theorem} (\cite{bartle}*{p.\ 195}) implies that improper R-integrals (and the same for improper gauge integrals) are \emph{automatically} gauge integrals.  
Thus, Hake's theorem implies that the gauge integral is (maximally) closed under improper integrals.  

\smallskip

We consider \emph{special cases} of Hake's theorem, including item~\eqref{zwaken} below which does mention gauge integrability \emph{but does not mention gauge integrals or their uniqueness}. As a result, it is fair to say that the following 
equivalences are not (only) based on the uniqueness of the gauge integral.  Note that Hake's theorem in general (see \cite{bartle}*{\S12 and \S16}) comes with {no restrictions}.  % on the class of functions other than that either $\int_{a}^{b}f$ or $\lim_{a\di c-}\int_{a}^{b}f$ exists.      
\begin{thm}\label{coromag} 
Over $\ACA_{0}^{\omega}+\QFAC^{2,1}$, the following are equivalent to $\HBU$:
\begin{enumerate}
\renewcommand{\theenumi}{\roman{enumi}}
\item There exists a function which is not gauge integrable with a modulus.\label{itemwon}
\item \(Hake\) If $f$ is gauge integrable on $I$ with a modulus and R-integrable on $[x,1]$ for $x>_{\R}0$, then the limit of R-integrals $\lim_{x\di 0+}\int_{x}^{1}f$ is $\int_{0}^{1}f$.\label{haken}
\item \(weak Hake\) If $f$ is gauge integrable with a modulus on $I$ and R-integrable on $[x,1]$ for $x>_{\R}0$, then the limit of R-integrals $\lim_{x\di 0+}\int_{x}^{1}f $ exists.\label{zwaken}
\end{enumerate}
\end{thm}
\begin{proof}
We shall prove $\HBU\di \eqref{haken}\di \eqref{zwaken}\di \eqref{itemwon} \di \HBU$.  Note that the second implication is trivial.   Now assume item \eqref{zwaken} and consider the function $g:I\di \R$ which is $0$ if $x=_{\R}0$, and $\frac{1}{x}$ otherwise.  This function exists in $\ACA_{0}^{\omega}$ by \cite{kohlenbach2}*{Prop.~3.12}.  By the development of integration theory in \cite{simpson2}*{IV.2}, the R-integral $\int_{x}^{1}g$ exists for $x>0$ and is readily seen to equal $\ln(x)$, the natural logarithm.  However, the limit $x\di 0+$ of this function is $-\infty$.  Thus, the limit $\lim_{x\di 0+}\int_{x}^{1}g $ does not exist, and by the contraposition of weak Hake's theorem, we conclude that $g$ is not gauge integrable with a modulus on $I$, i.e.\ item \eqref{itemwon} follows.  

\smallskip      

The implication $ \eqref{itemwon}\di \HBU$ follows from the proof of Theorem \ref{firstje}:  in the last part of the latter proof, it is shown
that $\neg\HBU$ allows us to define a gauge $\delta_{0}$ for which there are no $\delta_{0}$-fine partitions.
Hence, the underlined part in \eqref{corefu} is false, making the formula trivially true for any $f$ and $A$, i.e.\ every function is gauge integrable (with a modulus).
Contraposition now yields the desired implication.  

\smallskip

Finally, we prove item \eqref{haken} in $\ACAo+\HBU+\QFAC^{2,1}$ based on the proof in \cite{bartle}*{p.\ 195}.  
In a nutshell, the latter uses the \emph{Saks-Henstock lemma} to prove that the indefinite integral $F(x):=\int_{x}^{1}f$ is ($\eps$-$\delta$-)continuous in $x$ on $I$.        
Hence $\lim_{x\di0+}F(x)=_{\R}F(0)$, which is exactly as required for item \eqref{haken}.  First of all, the Saks-Henstock lemma intuitively states that if one considers a sub-partition 
of a $\delta$-fine partition, one inherets all the `nice' properties of the original partition.  The proof of this lemma is a straight-forward `epsilon-delta' argument, with one subtlety: the \emph{Cauchy criterion} (as is Theorem \ref{cauct}) for gauge integrals requires a gauge modulus, which we (therefore) assumed in item \eqref{haken}.  
The proof that the Saks-Henstock lemma yields the continuity of $F(x):=\int_{x}^{1}f$ is also a straight-forward `epsilon-delta' argument.    
\end{proof}
The gauge integral is a proper extension of the Lebesgue and (improper) Riemann integral.  
As it turns out, this claim is of considerable logical strength, as follows.  
%for the gauge integral to be \emph{non-trivial}, there should \emph{at least} be a function that is \emph{not} gauge integrable (with a modulus), 
%and 
\begin{cor}[$\RCAo+\WKL+\QFAC^{2,1}$]\label{ofvaluesee}
The below combination yields $\ATR_{0}$:
\begin{enumerate}
\item[(i)] There exists a function that is not gauge integrable with a modulus.  % but $|\kappa|$ is not.
\item[(ii)] There exists a function that is not Riemann integrable, but gauge integrable. 
\end{enumerate}
\end{cor}
\begin{proof}
By \cite{kohlenbach2}*{Prop.\ 3.7}, $\neg(\exists^{2})$ implies that all $F:\R\di\R$ are continuous, and hence uniformly continuous on $[0,1]$ by $\WKL$ and \cite{kohlenbach4}*{Prop.\ 4.10}.
Hence, any gauge integrable function is also Riemann integrable, as the gauge has an upper bound on $[0,1]$.  By contraposition, item (ii) from the theorem implies $(\exists^{2})$.  
By Theorem \ref{coromag}, item (i) now yields $\HBU$.  The combination $\HBU+(\exists^{2})$ yields $\ATR_{0}$ by \cite{dagsam}*{Cor.\ 6.7} and Theorem \ref{nolapdog}. 
\end{proof}

Fourth, we show that $\HBU$ is equivalent to the fact that the gauge integral is a proper extension of the Lebesgue integral.   
In fact, $f:[0,1]\di \R$ is Lebesgue integrable if and only if $|f|, f$ are gauge integrable (\cite{bartle}*{\S7, p.\ 102}).      
We use the latter variant as introducing the Lebesgue integral is beyond the scope of this paper.  
\begin{thm} Over $\ACA_{0}^{\omega}$, $\HBU$ is equivalent to the following statement:
 There exists a function $\kappa:I\di \R$ which is gauge integrable with a modulus but $|\kappa|$ is not.
\end{thm}
\begin{proof}
The reverse implication is immediate by Theorem \ref{coromag}.
For the forward implication, define $a_{k}:=1-\frac{1}{2^{k}}$ and $\kappa(x):=(-1)^{k+1}\frac{2^{k}}{k}$ if $x\in [a_{k-1}, a_{k})$ ($k^{0}\geq 1$), and $0$ otherwise. 
Then for $x>_{\R}0$, the area between the horizontal axis and the graph of $|\kappa|$ on $[0,x]$ is just a finite collection of (bounded) rectangles, i.e.\ $|\kappa|$ is definitely R-integrable on $[0,x]$ for $x<1$.
In particular, if $x\geq_{\R}1-\frac{1}{2^{k}}$, there are at least $k$ rectangles to the left of $x$; the first has base $1/2$ and area $1$, the second one base $1/4$ and area $1/2$, \dots, the $k$-th one has base $1/2^{k}$ and area $1/k$.  The R-integral $\int_{0}^{x}|\kappa|$ is thus at least $\sum_{i=1}^{k}\frac{1}{i}$.  The limit of the latter is the \emph{divergent} harmonic series, and item \eqref{zwaken} from Theorem \ref{coromag} yields that $|\kappa|$ is not gauge integrable on $I$ with a modulus.
To prove that $\kappa$ is gauge integrable on $I$, note that \eqref{hersing} allows us to `piece together' gauges.  
The following gauge modulus is based on that idea:
\[
\delta_{\eps}(x):=
\begin{cases}
d(x, E) & x \in [0,1]\setminus E  \\
\frac{\eps}{4^{(k+1)}}& x=_{\R}a_{k}\\
2^{-m(\eps)}&x=_{\R}1\\
\end{cases},
\]
where $E$ is the set consisting of the real $1$ and all $a_{k}$, and where $m(\eps)$ is such that $m(\eps)\geq \frac{1}{\eps}$ and the tail of the alternating harmonic series satisfies $|\sum_{k=n}^{\infty}\frac{-1^{k+1}}{k}|\leq \eps$ for $n\geq m(\eps)$.
We leave it as an exercise that this gauge can be defined in $\ACAo$.
The proof that $\delta_{\eps}$ is a gauge for $\kappa$ is completely straightforward and elementary, but somewhat long and tedious.  Hence, we omit this proof and refer to \cite{bartle}*{p.\ 35}. 
\end{proof}
%NEW&&&
Example \ref{splitskop} notwithstanding, we now prove that the base theory in Theorem~\ref{firstje} can be weakened to weak K\"onig's lemma. 
\begin{cor}
The equivalences in Theorem \ref{firstje} can be proved in $\RCAo+\WKL$. 
\end{cor}
\begin{proof}
As noted in the proof, only $\HBU\di \eqref{itemone}$ in Theorem \ref{firstje} requires $(\mu^{2})$.   
To prove this implication in $\RCAo+\WKL$, note that in case $(\mu^{2})$, we may use the proof of Theorem \ref{firstje}.
In case of $\neg(\mu^{2})$, all $\R\di \R$-functions are continuous by \cite{kohlenbach2}*{Prop.\~3.12}.  
Thanks to $\WKL$, all $\R\di \R$-functions are uniformly continuous on $[0,1]$ by \cite{kohlenbach4}*{Prop.\ 4.10}. 
Hence, the definition of gauge integral reduces to that of Riemann integral, and the latter is even unique in $\RCA_{0}$. 
The law of excluded middle as in $(\mu^{2})\vee \neg(\mu^{2})$ now finishes the proof.
\end{proof}
Fifth, we discuss in what sense we may evaluate (general) gauge integrals. 
\begin{rem}[Computing integrals]\label{engauged}\rm
In the case of the R-integral, a modulus (of R-integration) computes a $\delta>0$ in terms of any $\eps>0$ as in Definition \ref{GID}.  
Hence, if $P_{n}$ is the equidistant partition of $I$ with mesh $1/2^{n}$, we know that $S(P_{n}, f)$ converges to the R-integral of $f$ on $I$, and 
the modulus provides a rate of convergence.  For the gauge integral, there is no analogue of the equidistant partition: even given a gauge modulus $\delta(\eps, x)$, 
we need to find, say for every $\eps>0$, a $\delta(\eps, \cdot)$-fine partition $Q_{\eps}$; only then can we consider the limit of $S(Q_{\eps}, f)$ for $\eps\di 0$, which converges to 
the gauge integral of $f$ on $I$ as in Theorem \ref{cauct}.  To find such a partition, the only option (we can imagine) is to consider $\cup_{x\in I}(x-\delta(\eps, x), x+\delta(\eps, x))$ and apply the realiser $\Omega$ for $\HBU$ as in \eqref{1337} to obtain   
a finite sub-cover.  The latter can be modified using $\mu^{2}$ into a $\delta(\eps, \cdot)$-fine partition.  % and a realiser for $\HBU$ can be obtained from any $\Theta$ such that $\SCF(\Theta)$.  
\end{rem}
Sixth, two of the redeeming features of the gauge integral are its simplicity (via the similarity to the Riemann integral) and its generality.  
Indeed, the gauge integral boasts the most general version of the fundamental theorem of calculus and is `maximally' closed under improper integrals via Hake's theorem.  
We now consider the former theorem, which indeed only requires a modulus of differentiability.  
\begin{thm}[$\RCAo+\HBU$]\label{FTC}
Let $f:\R\di \R$ be differentiable with a modulus.   Then $\int_{a}^{b}f' = f(b)-f(a)$, and the latter modulus provides a gauge modulus. 
\end{thm}
\begin{proof}
The textbook proof (see e.g.\ \cite{bartle}*{p.\ 58}) amounts to little more than manipulation of definitions and of course goes through in $\RCAo+\HBU$. 
\end{proof}
It goes without saying that restricting the fundamental theorem of calculus to e.g.\ measurable functions (as is done in e.g.\ \cite{walshout}) flies in the face of the generality the gauge integral enjoys.  

\smallskip

Finally, we discuss some foundational implications of the above results.  
\begin{rem}\label{fefferketoch}\rm 
Feferman discusses in \cite{fefermanlight}*{V} what fragments of mathematics are necessary for the development of physics. 
He claims that the logical system \textbf{W}, a conservative extension of Peano arithmetic, suffices for this purpose.   
Feferman also discusses two purported (exotic) counterexamples involving non-measurable sets and non-separable spaces; he shows that these are rather fringe in physics, if part of the latter at all.  

\smallskip

However, by contrast, Feynman's path integral and the associated diagrammatic approach are central to large parts of modern physics.   To the best of our knowledge, 
the gauge integral is unique in that it provides a formalisation (of part) of Feynman's formalism that remains close to Feyman's development and ideas (based on Riemann sums), as discussed in \cite{mully}*{\S A.2} and \cite{burkdegardener}*{Ch.\ 10}.  
Moreover, the gauge integral avoids the non-physical concept of imaginary time (\cite{pouly,nopouly,nopouly2,nopouly3}). 
Hence, \emph{if} one requires that a mathematical formalisation remains close to (the original treatment in) physics, \emph{then} there seems to be no choice other than the      
gauge integral for the formalisation of Feynman's path integral.  As established in this section, the basic development of the gauge integral requires $\HBU$, and the latter is not provable in any $\SIXK$, a system much stronger than \textbf{W}.  Thus, Feferman's above claim seems incorrect, \emph{assuming} the aforementioned caveat concerning formalisations.    
\end{rem}
\subsubsection{Extra data and the gauge integral}\label{extradata}
We address a possible criticism that may be levelled at the results in the previous section.  
In particular, it is a natural question (from the pov of computability theory) whether adding `extra data' to the definition of the 
gauge integral leads to an integration theory in weaker systems. 

\smallskip

Before we answer this question, we emphasise that the definition of the gauge integral as in Definition \ref{GID} is taken verbatim from the literature, except the 
final item, the existence of which is however taken for granted throughout the literature.   Moreover, the first step in the development of the gauge integral is always the use of 
Cousin's lemma (or open-cover compactness) to show that the definition of the gauge integral `makes sense'.  For instance, Swartz writes the following.
\begin{quote}
In order for [the definition of the gauge integral] to make sense, there is one matter that needs to be
addressed. We must show that every gauge $\gamma$ has at least one $\gamma$-fine tagged partition. (\cite{zwette}*{p.\ 6})
\end{quote}
The next step is then always to show that the gauge integral is unique, and this proof is based on the existence of $\gamma$-fine partitions (see again \cite{zwette}*{p.\ 6}).
The same approach may be found in \cite{bartle, mullingitover}.  
Thus, the previous section deals with `actual mathematics' or `theorems as they stand', i.e.\ we have avoided the use of `extra data' to the maximal extent possible, 
as mandated by Simpson in \cite{simpson2}*{I.8.9}.  

\smallskip

\emph{Nonetheless}, for a variety of reasons, one may wish to `constructivise' the definition of the gauge integral, i.e.\ build in extra data to guarantee a development in weaker systems (than $\Z_{2}^{\Omega}$). 
We consider one possible/obvious such `richer' definition, and show that one still readily obtains $\HBU$.  
Now, the latter is needed in the proof of Theorem \ref{firstje} to guarantee the existence of $\delta$-fine partitions for arbitrary gauges $\delta$.  
Hence, it seems that the most obvious piece of `extra data' is the requirement that such a partition is given \emph{together with} the gauge, as follows.
\bdefi
A function $f:I\di \R$ is \emph{strongly gauge integrable} on $I$ if there is $A\in \R$ such that for all $\eps>0$, there is a gauge $\delta:\R\di \R^{+}$ such that  
\be\label{braaaak}
(\forall P\in \textsf{tp})(\textup{$P$ is $\delta$-fine }\di |S(f, P)-A|<_{\R}\eps)\wedge (\exists Q\in \textsf{tp})(\textup{$Q$ is $\delta$-fine}).
\ee
The notion of \emph{strong} gauge modulus is defined similarly.   
\edefi
Note that the second conjunct in \eqref{braaaak} guarantees that the antecedent in the first conjunct of \eqref{braaaak} cannot be vacuously true.  One may cherish the hope that this `extra data' obviates the use of $\HBU$, but we now 
show that the latter is still readily obtained from basic properties of the gauge integral.    

\smallskip

To this end, we now turn to the fundamental theorem of calculus for the gauge integral; this is the most general formulation of the well-known phenomenon that integration and differentiation cancel out.   
As noted above, the `standard' proof of the fundamental theorem of calculus in \cite{bartle}*{p.\ 58}, \cite{leekessjorsj}*{p.\ 6}, \cite{mullingitover}*{Ch.~2.4}, and \cite{zwette}*{Ch.~1} also 
establishes that a gauge modulus for the integral in $\int_{a}^{b}F' = F(b)-F(a)$ is simply any modulus of differentiability of $F$; this modulus provides the `delta' in terms of the `epsilon' and the point at which the derivative is taken.  
Similar constructs and results exist\footnote{As detailed in \cite{bish1}, a modulus function is part and parcel of the definition of a (uniformly) continuous and differentiable functions.  Any modulus of uniform continuity is also a modulus of Riemann integrability, as can be gleaned from \cite{bish1}*{p.\ 47}.} in e.g.\ Bishop's constructive analysis, and other computational approaches to mathematics.
\begin{thm}[$\FTC_{\textsf{strong}}$]
If $F:\R\di \R$ is differentiable with a modulus on $I$ and derivative $F'$, then the latter is strongly gauge integrable with the same modulus and we have $F(1)-F(0)=_{\R}\int_{0}^{1}F'$. 
\end{thm}
%Note that the derivative $F'$ in $\FTC$ need not be continuous, i.e.\ the Riemann integral would not be suitable, nor would the Lebesgue integral be.  
\begin{thm}\label{FUTC}
The system $\RCAo$ proves $\HBU\asa \FTC_{\textsf{\textup{strong}}}$.
\end{thm}
\begin{proof}
The forward direction follows by the usual proof from the literature;  in a nutshell, for the modulus of differentiability $\lambda x.\delta_{\eps}(x)$ of $F$,
%provides an uncountable cover $\cup_{x\in I}I^{\delta_{\eps}}_{x}$ of the unit interval, and 
any $\delta_{\eps}$-fine partition $P$ satisfies $|F(1)-F(0)-S(F', P)|<\eps$ using standard calculus tricks, known as `telescoping sums' and the `straddle lemma' (see e.g.\ \cite{bartle}*{p.\ 57-58}).
Note that $\HBU$ implies the equivalence between `normal' and strong gauge integrability. 

\smallskip

For the reverse direction, fix $\Psi:\R\di \R^{+}$ and pick any function $F:\R\di \R$ with derivative $F'$ and modulus of differentiability $G:\R^{2}\di \R$.  Then define $G_{0}(x, \eps)$ as the minimum of $G(x, \eps)$ and $\Psi(x)$.  
Clearly, $G_{0}$ is also a modulus of differentiability for $F$, and $\FTC_{\textsf{strong}}$ implies that there is a $G_{0}$-fine partition of the unit interval, immediately yielding a finite sub-cover for the canonical cover associated to $\Psi$. 
In this way, $\HBU$ follows, and we are done. 
\end{proof}
While the previous theorem is not particularly deep, it does suggest that a constructive treatment of the gauge integral is not entirely trivial. 

\subsection{Nets and compactness}\label{klipel}
We show that the Bolzano-Weierstrass theorem \emph{for nets} implies $\HBU$.  More results of this nature, e.g.\ for the monotone convergence theorem and the Dini and Arzel\`a theorems, may be found in \cite{samnet}.

\smallskip

The move to more and and more abstract mathematics can be quite concrete and specific: Moore presented a framework called \emph{General Analysis} at the 1908 ICM in Rome (\cite{mooreICM}) that was to be a `unifying abstract theory' for various parts of analysis.  
For instance, Moore's framework captures various limit notions in one abstract concept (\cites{moorelimit1}).  
This theory also included a generalisation of the concept of \emph{sequence} beyond countable index sets, nowadays called \emph{nets} or \emph{Moore-Smith sequences}.  
These were first described in \cite{moorelimit2} and formally introduced by Moore and his student Smith in \cites{moorsmidje}. %Our interest in nets stems from the elegant equivalent formulation of compactness they provide (see e.g.\ \cite{berkhof}*{p.\ 41}, \cite{moorsmidje}*{p.~118}, or \cite{ooskelly}*{p.\ 64}).
In this section, we derive $\HBU$ from the Bolzano-Weierstrass theorem for nets in $[0,1]$ and indexed by Baire space.  Since nets are a generalisation of sequences, the latter theorem thus provides a `unified' notion of compactness, implying both sequential and open-cover compactness.  

\smallskip  

We first need the following standard definition (see e.g.\ \cite{ooskelly}*{Ch.\ 2})
\bdefi[Nets]\label{nets}
A set $D\ne \emptyset$ with a binary relation `$\preceq$' is \emph{directed} if
\begin{enumerate}
 \renewcommand{\theenumi}{\alph{enumi}}
\item The relation $\preceq$ is transitive, i.e.\ $(\forall x, y, z\in D)([x\preceq y\wedge y\preceq z] \di x\preceq z )$.
\item The relation $\preceq$ is reflexive, i.e.\ $(\forall x\in D)(x\preceq x)$.
\item For $x, y \in D$, there is $z\in D$ such that $x\preceq z\wedge y\preceq z$.\label{bulk}
\end{enumerate}
For such $(D, \preceq)$ and topological space $X$, any mapping $x_{d}:D\di X$ is a \emph{net} in $X$.  
\edefi
The relation `$\preceq$' is often not mentioned together with the net; we also write $d_{1},\dots, d_{k}\succeq d$ as short for $(\forall i\leq k)(d_{i}\succeq d)$.  
\bdefi[Convergence of nets]
If $x_{d}$ is a net in $X$, we say that $x_{d}$ \emph{converges} to the limit $\lim_{d} x_{d}=y\in X$ if for every neighbourhood $U$ of $y$, there is $d\in D$ such that for all $e\succeq d$, $x_{e}\in U$. 
%Similarly, a point $y\in X$ is a \emph{limit point} of $x_{d}$ if for every neighbourhood, there is $d\in D$ such that $x_{d}\in U$.
\edefi
\bdefi[Sub-nets]\label{demisti}
A \emph{sub-net} of a net $x_{d}$ with directed set $(D, \preceq_{D})$, is a net $y_{b}$ with directed set $(B, \preceq_{B})$ such that there is a function $\phi : B \di D$ such that:
\begin{enumerate}
 \renewcommand{\theenumi}{\alph{enumi}}
\item the function $\phi$ satisfies $ y_{b} = x_{\phi(b)},$
\item $(\forall d\in D)(\exists b_{0}\in B)(\forall b\succeq_{B} b_{0})(\phi(b)\succeq_{D} d)$.
%?d?D,?b0 ?B such that if b\UTF{227B}b0 then ?(b)\UTF{227B}d.
\end{enumerate}
\edefi
In this section, we only study directed sets that are subsets of Baire space, i.e.\ as given by Definition \ref{strijker}.   
Similarly, we only study nets $x_{d}:D\di \R$ where $D$ is a subset of Baire space.  Thus, a net $x_{d}$ in $\R$ is just a type $1\di 1$ functional with extra 
structure on its domain $D$ provided by `$\preceq$' as in Definition \ref{strijker}.  
\bdefi[$\RCAo$]\label{strijker}
A `subset $D$ of $\N^{\N}$' is given by its characteristic function $F_{D}^{2}\leq_{2}1$, i.e.\ we write `$f\in D$' for $ F_{D}(f)=1$ for any $f\in \N^{\N}$.
A `binary relation $\preceq$ on a subset $D$ of $\N^{\N}$' is given by the associated characteristic function $G_{\preceq}^{(1\times 1)\di 0}$, i.e.\ we write `$f\preceq g$' for $G_{\preceq}(f, g)=1$ and any $f, g\in D$.
Assuming extensionality on the reals as in item \eqref{REXTJE} of Definition \ref{keepintireal}, we obtain characteristic functions that represent subsets of $\R$ and relations thereon.  
Using pairing functions, it is clear we can also represent sets of finite sequences (of reals), and relations thereon.  
\edefi
A basic result is that a topological space $X$ is compact if and only if every net in $X$ has a convergent sub-net. 
Let $\BW_{\net}$ be the Bolzano-Weierstrass theorem for nets, i.e.\ the statement that every net in the unit interval has a convergent sub-net, as can be found in e.g.\ \cite{reedafvegen}*{p.\ 98}.  
Note that $\BW_{\net}$ is assumed to be restricted as in Definition \ref{strijker}. 
We have the following theorem. 
\begin{thm}\label{merdes}
The system $\RCAo+\BW_{\net}$ proves $\HBU$.
\end{thm}
\begin{proof}
Note that $\BW_{\net}$ implies the monotone convergence theorem, as sequences are clearly nets.  Hence, we have access to 
$\ACA_{0}$ by \cite{simpson2}*{III.2.2}.  Now, in case $\neg(\exists^{2})$, all functions on $\R$ are continuous by \cite{kohlenbach2}*{Prop.\ 3.12}, and $\HBU$ reduces to $\WKL$ by \cite{kohlenbach4}*{\S4}.
We now prove $\HBU$ in case $(\exists^{2})$, which finishes the proof using the law of excluded middle.
Thus, suppose $\neg\HBU$ and fix some $\Psi:I\di \R^{+}$ for which $\cup_{x\in I}I_{x}^{\Psi}$ does not have a finite sub-cover.
Let $D$ be the set of all finite sequences of reals in the unit interval, and define `$v \preceq_{D} w$' for $w, v\in D$ if $\cup_{i<|v|}I_{v(i)}^{\psi}\subseteq \cup_{i<|w|}I_{w(i)}^{\psi}$, i.e.\ 
the set generated by $w$ includes the set generated by $v$.  Note that $(\exists^{2})$ suffices to define $\preceq_{D}$.  Clearly, the latter is transitive and reflexiv, also satisfies 
item \eqref{bulk} in Definition \ref{nets}.  To define a net, consider 
\be\label{okl}
(\forall w^{1^{*}}\in [0,1])(\exists q \in \Q\cap [0,1])\underline{(q\not \in \cup_{i<|w|}I_{w(i)}^{\Psi})}, 
\ee
which again holds by assumption.  Note that the underlined formula in \eqref{okl} is decidable thanks to $(\exists^{2})$.  
Applying $\QFAC^{1,0}$ to \eqref{okl}, we obtain a net $x_{w}$ in $[0,1]$, which has a convergent (say to $z_{0}\in I$) sub-net $y_{b}=x_{\phi(b)}$ for some directed set $(B, \preceq_{B})$ and $\phi:B\di D$, by $\BW_{\net}$.
By definition, the neighbourhood $U_{0}=I_{z_{0}}^{\Psi}$ contains all $y_{b}$ for $b\succeq_{B} b_{1}$ for some $b_{1}\in B$.  However, taking $d=\langle z_{0} \rangle$ in the second item in Definition \ref{demisti}, there is also $b_{0}\in B$ such that $(\forall b\succeq_{B} b_{0})(\phi(b)\succeq_{D} \langle z_{0}\rangle)$.
By the definition of `$\preceq_{B}$', $\phi(b)$ is hence such that $\cup_{i<|\phi(b)|}I_{\phi(b)(i)}^{\Psi}$ contains $U_{0}$, for any $b\succeq_{B} b_{0}$.  Now use item~\eqref{bulk} from Definition~\ref{nets} to find $b_{2}\in B$ satisfying $b_{2}\succeq_{B} b_{0}$ and $b_{2}\succeq_{B} b_{1}$.  
Hence, $y_{b_{2}}=x_{\phi(b_{2})}$ is in $U_{0}$, but $\cup_{i<|\phi(b_{2})|} I^{\Psi}_{\phi(b_{2})(i)}$ also contains $U_{0}$, i.e.\ $x_{\phi(b_{2})}$ must be \emph{outside} of $U_{0}$ by the definition of $x_{w}$, a contradiction. 
In this way, we obtain $\HBU$ in case $(\exists^{2})$, and we are done.
\end{proof}
Finally, we note that Moore and Smith already proved versions of the Bolzano-Weierstrass, Dini, and Arzel\`a theorem in \cite{moorsmidje}.

\section{Main results II}

\subsection{Jumping to the superjump}\label{RMR2}
We show that the Lindel\"of lemma \emph{for Baire space} and Feferman's $\mu^{2}$ together give rise to the Suslin functional $S$ and the \emph{superjump} $\SJ$.  
We introduce the latter in Section \ref{pofu}, while the Lindel\"of lemma \emph{for Baire space}
 and the associated functional $\Xi$ (computing the countable sub-cover) are introduced in Section~\ref{lindeb}.  
The following results are established below.
\begin{enumerate}
\item The superjump $\SJ$ is computable in the special fan functional $\Theta$ and the Suslin functional $S$ (Section \ref{pofu}).
\item The Suslin functional $S$ is (uniformly) computable in Feferman's $\mu$ and the functional $\Xi$ which computes the countable sub-cover from the Lindel\"of lemma for Baire space (Section \ref{lindeb}).
\end{enumerate}
As a consequence, the combination of Feferman's $\mu$ and any such $\Xi$ computes the superjump $\SJ$.
We recall the fact that the special fan functional $\Theta$ is not unique, and neither is `the' aforementioned functional $\Xi$.

\subsubsection{Computing the superjump}\label{pofu}
We show that the combination of the Suslin functional $S$ and the special fan functional $\Theta$ computes the \emph{superjump}.  
The latter corresponds to the Halting problem for computability on type two inputs.  Indeed, the superjump $\SJ^{3}$ was introduced in \cite{supergandy} by Gandy (essentially) as follows:
\be\tag{$\SJ^{3}$}
\SJ(F^{2},e^{0}):=
\begin{cases}
0 & \textup{ if $\{e\}(F)$ terminates}\\
1 & \textup{otherwise}
\end{cases},
\ee
where the formula `$\{e\}(F)$ terminates' is a $\Pi_{1}^{1}$-formula defined by Kleene's S1-S9.

\smallskip

As to its history, Harrington has proved that the first ordinal not computable in $\SJ$ is the first recursively Mahlo ordinal (\cite{superharry}).  In turn, the latter ordinal appears in the study of constructive set and type theory and the associated proof theory (\cite{manom, manom2, manom3}).  In particular, $\{R\subseteq \N: R \textup{ is computable from } \SJ\}$ is the smallest $\beta$-model of $\Delta_{2}^{1}\textsf{-CA}_{0}+\textsf{(M)}$, where \textsf{(M)} expresses that every true $\Pi_{3}^{1}$-sentence with parameters already holds in a $\beta$-model of $\Delta_{2}^{1}$-comprehension (\cite{manom}).   
As discussed in Remark~\ref{predifiel}, $\SJ$ lives far outside of predicative mathematics.
\begin{thm}\label{super} %Let $\SCF(\Theta^3)$.
The superjump $\SJ$ is computable in any $\Theta$ satisfying $\SCF(\Theta)$ and the Suslin functional $S$.
\end{thm}
\begin{proof}
We first provide a sketch of the proof as follows. Recall that if $\sigma$ is a finite binary sequence, then $[\sigma]$ is the set of total binary extensions of $\sigma$.
\begin{enumerate}
\item Given $F^{2}$, let $\alpha_F(e) = \{e\}(F,e)$ whenever the value is in $\{0,1\}$, and let $X_F$ be the set of total binary extensions of $\alpha_F$.
\item Compute $G_F$ from $F$ and $S$ with the properties
\begin{itemize}
\item[i)]  if $f \not \in X$, then  $G_F(f) > 0$
\item[ii)] if $G_F(f) > 0$, then $[\bar fG_F(f)]$ does not intersect $X_F$
\item[iii)] for $f \in X_F$, $G_F(f) = 0$.
\end{itemize} 
\item Show that $\SJ(F)$ is uniformly computable in $S$ and $f \in X_{F}$.
\item Since $\Theta(G_F)$ has to intersect $X_F$, and we can decide where, $\SJ(F)$ is  computable in $\Theta$ and $S$, uniformly in $F$.
\end{enumerate}
We work out the proof in full detail below.
\end{proof}
We will now list some basic lemmas needed for the detailed proof of Theorem~\ref{super}. We first define an important concept relating to S1-S9 computability with type two inputs.   Its importance stems from the fact  that it is independent of the choice of input functional $F^{2}$, as follows.
\begin{lem}There is a primitive recursive $\xi$ of type level 1, independent of the choice of $F^{2}$, such that  
$\{\xi(e, \vec a)\}(F,\xi(e,\vec a))$ is resp.\ \($0, 1 $,   {\em undefined}\) whenever $\{e\}(F,\vec a)$ is resp.\ \($= 0$,  $> 0$,  {\em undefined}\). 
\end{lem}
\begin{lem}
There is a primitive recursive function $\eta$ such that for all $e,\vec a, F$
\[
\{\eta(e)\}(F,\vec a) \simeq \{e\}(F,\vec a) \dminus 1,
\]
where `$\simeq$' means that both sides are undefined or both sides are defined and equal.
\end{lem}
\begin{defi}\rm
Let $f$ be a total binary function. By an application of the recursion theorem for Turing computations in oracles we define
\[
[e]_f(\vec a) := 
\begin{cases} 
0& \textup{if $f(\xi(e,\vec a)) = 0$} \\ 
1 + [\eta(e)]_f(\vec a)&\textup{if $f(\xi(e,\vec a)) = 1$}
\end{cases}.
\]
Clearly, if the recursion goes on forever, $[e]_f(\vec a)$ will be undefined. 
\end{defi}
Intuitively speaking and from the outside, $[ \cdot]_f$ may look like an indexing of some  partial functions computable in some functional of type 2, but to what extent this is correct, will depend on the choice of $f$. 

\smallskip

We will now use $F$ to define a relation, mimicking the subcomputation relation relative to $F$, as far as possible. As a cheap trick, we will let an alleged computation tuple be a subcomputation of its own if it is clear that something is wrong, in order to force such objects into the non-well-founded part of the relation.
\begin{defi}\label{formdix}\rm
Given $f$, we let $\Omega_f$ be the set of triples  $(e,\vec a,b)$ such that $[e]_f(\vec a) = b$. Given $F$ as well, define the relation `$\preceq$' \(short for $\preceq_{f,F}$\) on $\Omega_f$ as follows:
\begin{itemize}
\item If $e$ is not a Kleene index for any of S1-S9, we put $(e,\vec a , b) \preceq (e , \vec a , b)$.
\item If $e$ is an index for an initial computation, we let $(e,\vec a , b)$ be a leaf in our ordering if $\{e\}(F,\vec a) = b$, and its own sub-node otherwise. 
This decision will be independent of the choice of the functional $F$.
\item We treat the case S4. The rest of the cases, except S8, are similar or void (e.g.\ S6).
If $e$ is an index for composition  
$\{e\}(F,\vec a) = \{e_1\}(\{e_2\}(\vec a),\vec a)$, $c$ is given and there is a $b$ such that
$ [e_2]_f(\vec a) = b$, $[e_1]_f(b,\vec a) = c$ and $[e]_f(\vec a) = c$,
then we define $(e_2,\vec a , b) \preceq (e,\vec a , c)$ and $(e_1, b,\vec a,c) \preceq (e,\vec a, c)$.  If there is no such $b$ , we let $(e,\vec a, c) \preceq (e,\vec a, c)$.
\item For the case S8, if we have $\{e\}(F,\vec a) = F(\lambda b.\{d\}(F,b,\vec a))$, we let $(e,\vec a, c) \preceq (e,\vec a, c)$ unless $h(b) = [d]_f(b,\vec a)$ is a total function and  $F(h) = c$.
In the latter case, we let $(d,b,\vec a , h(b)) \preceq (e,\vec a , c)$ for all $b$.
\end{itemize}
\end{defi}
The intuitive explanation of Definition \ref{formdix} is as follows: The set of finite sequences $(e,\vec a , b)$ such that $\{e\}(F,\vec a  ) = b$ is defined by a strictly positive inductive definition, so whenever a sequence is added to the set it is either initial or there is a unique set of other sequences in the set causing that we accept the one chosen. These are called \emph{immediate} predecessors in the computation tree. The relation `$\preceq$' is defined on the set of $(e,\vec a , b)$ where $[e]_f(\vec a) = b$ as the immediate predecessor relation wherever the inductive definition of the computation tree is locally correct. 
\begin{lem} \label{lemma1.6.SJ}
For any function $f$, the well-founded segment of $\langle \Omega_f , \preceq_{f,F}\rangle$ is an initial segment of the full computation relation of $F$.
\end{lem} 
\begin{proof}
This is trivial by induction over this well-founded segment.
\end{proof}
\begin{lem} 
For any $f \in X_F$, if $\{e\}(F,\vec a) = b$, then $[e]_f(\vec a) = b$. 
\end{lem}
\begin{proof}
We prove this by induction on $b$. 
If $b = 0$, then $\{\xi(e,\vec a)\}(F,\xi(e, \vec a)) = 0$, so $f(\xi(e,\vec a)) = 0 = [e]_f(\vec a)$.
If $b > 0$, we use the induction hypothesis on $b \dminus 1$ for the index $\eta(e)$ and the fact that $[e]_f(\vec a) = b$ in this case.  
\end{proof}
\begin{lem} 
If $f \in X_F$ and $\{e\}(F,\vec a) = b$, then $(e,\vec a , b)$ is in the $\preceq_{f,F}$-well-founded part of $\Omega_f$. Moreover, this well-founded part is exactly the full tree of terminating computations $\{e\}(F,\vec a) = b$ relative to $F$. 
\end{lem}
\begin{proof}
That the computation tree for computations relative to $F$ is contained in the well-founded part is proved by induction over the tree of real computations. 
Now, if the well-founded part of $\langle \Omega_f,\preceq_{f,F}\rangle$ contains more, we may consider one alleged computation $(e,\vec a , b)$ in $\Omega_f$ that is not a real $F$-computation, but that is minimal as such. Since it is in the well-founded part, $(e,\vec a , b)$ is locally correct, so either it is an initial computation or it has subcomputations that are real (because we consider a minimal one).  Being locally correct, we see in each case that $(e,\vec a , b)$ must be genuine after all. 
\end{proof}
\begin{lem} 
If $f \in X_F$, then $\SJ(F)$ is uniformly computable in $f$, $F$ and $S$. 
\end{lem}
\begin{proof}
From the data, we can compute the characteristic function of $\{(e,\vec a , b) \mid \{e\}(\vec a) = b\}$, and $\SJ(F)$ is primitive recursive in this characteristic function. 
\end{proof}
We are now ready to provide the proof of Theorem \ref{super} as follows.  
\begin{proof}
We see from Lemma \ref{lemma1.6.SJ} that if the $\preceq_{f,F}$-well-founded part of $\Omega_f$ is closed under the Kleene schemes S1-S9 relative to $F$, then $\SJ(F)$ is computable in $f$, $F$ and $S$ as above. We need $S$ to isolate the well-founded part, and (only) $F$ and $\mu^{2}$ to decide if we have the closure. 

\smallskip

Now, assume that $f$ is such that the $\preceq_{f,F}$-well-founded part is not S1-S9-closed. Let $\{e\}(F,\vec a) = b$ be a computation of minimal rank such that we do not have $[e]_f(\vec a) = b$. By induction on $b$ we see that there must be an index $d$ such that $\{d\}(F,d) \in \{0,1\}$ and $\{d\}(F,d) \neq f(d)$. If we then put $G_F(f):= d+1$ we have ensured that there will be no extension of $\overline f G_F(f)$ in $X_F$. Using Gandy selection for $F,\mu$ and $f$, we can trivially find a $d$ with this property from the well-founded part of $\Omega_f$.  
In order to show that $G_F$ is definable from $S,F,\mu$ via a term in G\"odel's $T$, we proceed as follows:
\begin{quote}
Given the well-founded part $W$ of $\Omega_F$, we may arithmetically decide if it respects S1-S9. If it does not, let $\Gamma$ be the, arithmetically in $F$, inductive definition of the computation tuples for computing relative to $F$, and by one application of $\mu$ on $\Gamma(W) \setminus W$, we may find the $(e,\vec a , b)$ that leads us to the $d$ we need.
\end{quote}
In light of the previous, we put $G_{F}(f) := 0$ if the $\preceq_{f,F}$-well-founded part of $\Omega_f$ is a fixed point of the inductive definition of computations relative to $F$, while we put $G_{F}(f): = d+1$ for the $d$ selected as above otherwise.  Thus,  $\Theta(G_F)$ must contain a function from which, together with $F$ and $S$, we can compute $\SJ(F)$. 
\end{proof}
Some of the methods in this proof have been expanded in \cite{dagsamV, dagcie18}, where even sharper results are obtained.

\subsubsection{Computing the Suslin functional}\label{lindeb}
We show that the Suslin functional $S$ can be computed by the combination of Feferman's $\mu$ and the functional $\Xi$ arising from the Lindel\"of lemma for $\N^{\N}$.  
Furthermore, the Lindel\"of lemma for $\N^{\N}$ (not involving $\Xi$) and the axiom $(\mu^{2})$ are seen to imply $\FIVE$. 

\smallskip

Regarding the Linde\"of lemma, we recall that Lindel\"of already proved that Euclidean space is \emph{hereditarily Lindel\"of} in \cite{blindeloef} around 1903.  
Now, the latter hereditary property implies that $\N^{\N}$ has the Lindel\"of property, since $\N^\N$ is homeomorphic to the irrationals in $[0,1]$ using continued fractions expansion.
Thus, for any $\Psi^{2}$, the corresponding `canonical cover' of $\N^{\N}$ is $\cup_{f\in \N^{\N}}\big[\overline{f}\Psi(f)\big]$ where $[\sigma^{0^{*}}]$ is the set of all extensions in $\N^{\N}$ of $\sigma$.  By the Lindel\"of lemma for $\N^{\N}$, there is a sequence $f_{(\cdot)}^{0\di 1}$ such that the set of $\cup_{i\in \N}[\bar f_{i} \Psi(f_i)]$ still covers $\N^{\N}$, i.e.\
\be\tag{$\LIND_{4}$}
(\forall \Psi^{2})(\exists f_{(\cdot)}^{0\di 1})(\forall g^{1})(\exists n^{0})( g \in  \big[\overline{f_{n}}\Psi(f_{n})\big] ).
\ee
Similar to the specification $\SCF(\Theta)$ for the special fan functional $\Theta$, we introduce the following specification based on $\LIND_{4}$. 
As for the former specification, the functional $\Xi^{2\di (0\di 1)}$ satisfying $\LIN(\Xi)$ is not unique.  
\be\tag{$\LIN(\Xi)$}
(\forall \Psi^{2})(\forall g^{1})(\exists n^{0})( g \in  \big[\overline{\Xi(\Psi)(n)}\Psi(\Xi(\Psi)(n))\big] ).
\ee
As for the special fan functional $\Theta$ in Theorem \ref{nolapdog}, the existence of $\Xi$ as in $\LIN(\Xi)$ amounts to the Lindel\"of lemma $\LIND_{4}$ itself.  
\begin{thm}\label{cruckZ}
The system $\FIVE^{\omega}+\QFAC^{2,1}$ proves $\LIND_{4}\asa (\exists \Xi)\LIN(\Xi)$.  
\end{thm}
\begin{proof}
We only need to prove the forward direction. We rephrase $\LIND_{4}$ to
\be\label{arguko}
(\forall G^{2})(\exists f_{(\cdot)}^{0\di 1})\big[(\forall g^{1})(\exists n^{0})( g \in  \big[\overline{f^+_{n}}f_n(0)\big]) \wedge (\forall m^{0})(f_m(0) = G(f_m^+))\big],
\ee
where $f^+(k) = f(k+1)$.
Using the Suslin functional $S$ and $\mu$ we see that the part of \eqref{arguko} inside the (outermost) square brackets can be viewed as quantifier-free, and thus the existence of $\Xi$ follows from $\QFAC^{2,1}$.  
\end{proof}
\noindent
For a (much) weaker base theory, we need the following functional from \cite{dagsam, dagsamV}.  
%We will see that $\exists^3$, or an equivalent $\kappa^3$ defined just on $C$, will be computable in any $M_o$ satisfying $\WPR$.
\be\tag{$\kappa_{0}^{3}$}
(\exists \kappa_{0}^{3}\leq_{3}1)(\forall Y^{2})\big[\kappa_{0}(Y)=0\asa (\exists f\in C)Y(f)=0  \big].
\ee
Here, $\RCAo+\WKL+(\kappa_{0}^{3})+\QFAC^{0,1}$ is conservative
%\footnote{To be absolutely clear, we take `$\WKL$' to be the $\L_{2}$-sentence \emph{every infinite binary tree has a path} as in \cite{simpson2}, while the Big Five system $\WKL_{0}$ is $\RCA_{0}+\WKL$, and $\WKL_{0}^{\omega}$ is $\RCAo+\WKL$.} 
\emph{up to language}\footnote{The fundamental objects in the language of $\RCAo$ are functions, with sets being definable from these, while it is exactly the opposite for $\RCA_{0}$. This however makes no difference.} over $\WKL_{0}$ by \cite{kohlenbach2}*{Prop.\ 3.15}, while $\RCAo$ proves that $[(\exists^{2})+(\kappa_{0}^{3})]\asa (\exists^{3})$ by \cite{dagsam}*{Rem.\ 6.13}.  
\begin{cor}\label{cruckZ222}
The system $\RCAo+(\kappa_{0}^{3})+\QFAC^{2,1}$ proves $\LIND_{4}\asa (\exists \Xi)\LIN(\Xi)$.  
\end{cor}
\begin{proof}
Consider $(\exists^{2})\vee \neg(\exists^{2})$ and note that $\exists^{3}$ follows in the former case, while all functions on Baire space are continuous in the latter case by \cite{kohlenbach2}*{Prop.\ 3.7}.
Hence, we may define $\Xi^{3}$ in the latter case as the functional that lists all finite sequences of natural numbers on input any (by assumption continuous) functional.  
\end{proof}
The functional $\Xi$ is weak in insolation, by the following theorem.  
\begin{thm}\label{predinot}
$\RCAo+(\exists \Xi)\LIN(\Xi)$ proves the same $\L_{2}$-sentences as $\RCA_{0}$.  
\end{thm}  
\begin{proof}
As in the proof of Corollary \ref{dirfi}, it suffices to show that $[(\exists \Xi)\LIN(\Xi)]_{\ECF}$ is provable in $\RCA_{0}$.  
However, $(\exists \Xi)\LIN(\Xi)$ only involves objects of type $0$ and $1$ except for the two leading quantifiers.  
Hence, $[(\exists \Xi)\LIN(\Xi)]_{\ECF}$ is as follows:
\[
(\exists \xi^{1}\in K_{0})(\forall \gamma^{1}\in K_{0})(\forall g^{1})(\exists n^{0})( g \in  \big[\overline{\xi(\gamma)(n)}\gamma(\xi(\gamma)(n))\big] ).
\]
Thus, by defining $\xi$ as the enumeration of $\gamma(w)$ as in the proof of Theorem \ref{main2}, we obtain an associate for a functional producing a countable sub-cover, and the sentence $[(\exists \Xi)\LIN(\Xi)]_{\ECF}$ is therefore provable in $\RCA_{0}$.    
\end{proof}
The functional $\Xi$ becomes strong when combined with $\mu^{2}$, as follows.  
\begin{thm}\label{super2}
The Suslin functional $S$ is uniformly computable in Feferman's $\mu$ and any $\Xi$ satisfying $\LIN(\Xi)$.
Furthermore, $\ACAo+(\exists \Xi)\LIN(\Xi)$ proves $(S^{2})$.
\end{thm}
\begin{proof}
Recall the definition of the Suslin functional $S$ as follows:
\[
S(f) = \left\{ \begin{array}{cl} 0&{\rm if~} (\exists g^1) (\forall n^0) (f(\bar gn) = 0) \\ 1 & {\rm otherwise}\end{array}\right.~.
\]
Define $F^2_f(g)$ as $n+1$ if $n$ is minimal such that $f(\bar gn) > 0$,  and $0$ if there is no such $n$.
Note that $F_f$ is readily defined from $f$ (in terms of $\mu^{2}$) inside $\ACAo$, and note that if $F_f(h) > 0$ and $\bar gF_f(h) = \bar hF_f(h)$, then $F_f(g) = F_f(h)$.
Let $\Xi$ be such that $\LIN(\Xi)$, and consider the following formula  
\be\label{durkio}
S(f) = 0 \asa (\exists i^{0} )(F_f(\Xi(F_{f})(i)) = 0).
\ee
The reverse direction in \eqref{durkio} is immediate by the definition of $F_{f}$.  For the forward direction, assume $S(f)=0$ and let $g^{1}$ satisfy $(\forall n^0)(f(\bar gn) = 0)$, i.e.\ $F_f(g) = 0$. As observed above, if $F_f(h) > 0$, we have  $g \not\in [\bar h F_f(h)]$; hence if $F_f(h_n) > 0$ \emph{for all $n\in \N$} where $h_{n} = \Xi(F_f)(n)$, the corresponding countable subset of the covering induced by $F_f$ will not be a covering.  Thus $F_f(\Xi(F_f)(n)) = 0$ must hold for some $n$, i.e.\ the right-hand side of \eqref{durkio} follows.
Finally, \eqref{durkio} clearly characterises $S(f)$ in terms of $\mu$, $f$ and $\Xi$ (via a term in G\"{o}del's $T$), and we are done.
\end{proof}
%NEW&&
The reader can readily verify that the proof in the theorem also goes through using intuitionistic logic.  
Combining the previous results, we get the following.  % nice results. 
\begin{cor}
$\RCAo+\QFAC^{2,1}$ proves $[(S^{2})+\LIND_{4}]\asa [(\exists \Xi)\LIN(\Xi)+(\mu^{2})]$.
\end{cor}
%Combining the above results, we obtain the following. 
\begin{cor}\label{superduper}
The superjump $\SJ$ is computable in any $\Xi$ satisfying $\LIN(\Xi)$ and Feferman's $\mu$, by a term in G\"{o}del's $T$.
\end{cor}
\begin{proof} Given such $\Xi$, there are terms $t_1, t_2$ such that $\SCF(t_1(\Xi,\mu))$ (i.e.\ $\Theta$ is given by $t_1(\Xi,\mu)$), and $S =_{2} t_2(\Xi,\mu)$. Checking the details  of the proof of Theorem \ref{super} and the construction of $G_F$, we see that there is a term $t_3$ such that $G_F(f) = t_3(F,f,S,\mu)$. Since $\SJ(F)$ is primitive recursive in $\Theta(G_F)$, the theorem follows.
\end{proof}
Finally, the presence of $\Xi$ is not necessary if one is only interested in $\FIVE$.
In particular, the following version of the Lindel\"of lemma expresses that for a \emph{sequence} of open covers of Baire space, there is a \emph{sequence} of countable sub-covers. 
\be\tag{$\LIND_{\seq}$}
(\forall \Psi_{(\cdot)}^{0\di 2})(\exists f_{(\cdot, \cdot)}^{(0\times 0)\di 1})(\forall m^{0})\big[(\forall g^{1})(\exists n^{0})\big( g \in  [\overline{f_{n,m}}\Psi_{m}(f_{n,m})] \big)\big].
\ee
Note that such `sequential' theorems are well-studied in RM, starting with \cite{simpson2}*{IV.2.12}, and can also be found in e.g.\ \cites{fuji1,fuji2,hirstseq,dork2,dork3}.  The following corollary was first published in \cite{dagsamV}.  
\begin{cor}\label{theultimate}
The system $\ACAo+\LIND_{\seq}$ proves $\FIVE$.
\end{cor}
\begin{proof}
The proof of Theorem \ref{super2} goes through with minor modification.  
Due to the below `grand claims' based on this corollary, we do provide the proof in some detail.  
First of all, by \cite{simpson2}*{V.1.4}, any $\Sigma_{1}^{1}$-formula can be brought into the `normal form' $(\exists g^{1})(\forall n^{0})(f(\overline{g}n) =0)$, given arithmetical comprehension. 
Thus, suppose $\varphi(m)\in \Sigma_{1}^{1}$ has normal form $(\exists g^{1})(\forall n^{0})(f(\overline{g}n,m) =0)$ and define $F^2_m$ as follows: $F_{m}(g)$ is $n+1$ if $n$ is minimal such that $f(\bar gn,m) > 0$, and $0$ if there is no such $n$.
Note that $F_{m}$ is based on $F_{f}$ from the theorem.  Apply $\LIND_{\seq}$ for $\Psi^{2}_{(\cdot)}=F_{(\cdot)}$ and let $f_{(\cdot, \cdot)}$ be the sequence thus obtained.  We define $X\subset\N$ as follows:
\be\label{jawell}
X:=\{ m^{0} : (\exists n^{0})(F_{m}(f_{n,m})=0)   \}, 
\ee
using $(\mu^{2})$.  We now prove $(\forall m^{0})(m\in X\asa \varphi(m))$, establishing the corollary.  If $m\in X$, then there is $g^{1}$ such that $F_{m}(g)=0$, i.e.\ $(\forall n^{0})(f(\overline{g}n,m)=0)$ by definition, and hence $\varphi(m)$.  
Now assume $\varphi(m_{0})$ for fixed $m_{0}$, i.e.\ let $g_{0}$ be such that $(\forall n^{0})(f(\overline{g_{0}}n,m_{0}) =0)$, and note that for any $m^{0}, g^{1}, h^{1}$, if $F_m(h) > 0$ and $\bar gF_m(h) = \bar hF_m(h)$, then $F_m(g) = F_m(h)$.
In particular, if $F_{m_{0}}(h) > 0$, we have  $g_{0} \not\in [\bar h F_{m_{0}}(h)]$.  Hence, if $F_{m_{0}}(f_{n,m_{0}}) > 0$ \emph{for all $n^{0}$}, $g_{0}$ is not in the covering consisting of the union of $[\overline{f_{n,m_{0}}}F_{m_{0}}(f_{n,m_{0}})]$ for all $n^{0}$, contradicting $\LIND_{\seq}$.  
Thus, we must have $(\exists n^{0})(F_{m_{0}}(f_{n,m_{0}})=0)$, implying that $m_{0}\in X$ by \eqref{jawell}. 
%Thus $F_f(\Xi(F_f)(n)) = 0$ must hold for some $n$, i.e.\ the right-hand side of \eqref{durkio} follows.
\end{proof}
Due to he fact that $\N \times \N^\N$ is trivially homeomorphic to $\N^\N$, $\LIND_\seq$ is derivable from (and hence equivalent to) $\LIND_4$, and we obtain the following result.
\begin{cor}\label{jaweltoch}
The system $\ACAo+\LIND_{4}$ proves $\FIVE$. 
\end{cor}

\begin{rem}[On predicativist mathematics]\label{predifiel}\rm
We have discussed the \emph{compatibility problem} for Nelson's predicative arithmetic (and its negative answer) in Section~\ref{comp}.  
We now argue that Corollary \ref{jaweltoch} also settles the compatibility problem for Weyl-Feferman predicative mathematics \emph{in the negative}.  
To this end, we exhibit two \emph{natural} theorems $A$ and $B$ which are both \emph{acceptable} in predicative mathematics but $A\wedge B$ is not.  
In a nutshell, $\ATR_{0}$ is considered the `upper limit' of predicative mathematics; both $\RCAo+\LIND_{4}$ and $\ACAo$ fall `well below' this upper limit, while the combination $\ACAo+\LIND_{4}$ falls `well above' the upper limit.  Hence, $\ACAo$ and $\RCAo+\LIND_{4}$ are acceptable in predicative mathematics, but the combination is not: $\FIVE$ is even the textbook example of an impredicative system (\cite{simpson2}*{\S I.12}).  
A detailed discussion (including technicalities) is as follows. 

\smallskip

First of all, we elaborate on the notion of `acceptable in predicative mathematics'.
On one hand there is Feferman's notion of \emph{predicative provability} (\cite{fefermanga, fefermanmain}), 
which is rather limited and clumsy when dealing with ordinary mathematics, according to Simpson (\cite{simpsonfriedman}*{p.\ 154}).
On the other hand, the weaker notion of \emph{predicative reducibility} is more flexible: a formal system $T$ is \emph{predicatively reducible} 
if -intuitively speaking- it is not stronger than a system $S$ which is predicatively provable.  
Thus, while $T$ may involve \emph{impredicative} notions, the latter are `safe' from the point of view of predicative mathematics as these notions only provide as much strength/power as $S$, and the latter's `predicative status' is well-known.    

\smallskip

Secondly, Feferman and Sch\"utte have shown (independently) that the least \emph{non-predicatively provable} ordinal is $\Gamma_{0}$ (See \cite{fefermanmain}*{p.\ 607} for details and references).  Hence, a formal system $T$ is called \emph{predicatively reducible} if its ordinal $|T|$ satisfies $|T|<\Gamma_{0}$.   
Note that $|\ATR_{0}|=\Gamma_{0}$, which motivates the status of $\ATR_{0}$ as the upper limit of predicative mathematics.  
Now, the proof-theoretic ordinal of $\RCAo+\LIND_{4}$ (resp.\ $\ACAo$) is $\omega^{\omega}$ (resp.\ $\eps_{0}$) by Theorem \ref{predinot} (resp.\ \cite{yamayamaharehare}*{Theorem 2.2}) and \cite{simpson2}*{IX.5}.
Since $\omega^{\omega}<\eps_{0}<\Gamma_{0}$, both these systems are predicatively reducible.  
By contrast, the combination of these systems, namely $\ACAo+\LIND_{4}$ implies $\FIVE$ by Corollary \ref{jaweltoch}, and the ordinal for the latter system is far beyond $\Gamma_{0}$.  
We refer to \cite{simpson2}*{IX.5} for background concerning the cited results and further references.  

\smallskip

Finally, we believe there to be many \emph{purely logical} statements $C$ and $D$ that are predicatively reducible, while $C\wedge D$ is not.  
Nonetheless, to the best of our knowledge, our result as in Corollary \ref{jaweltoch} is unique in that it provides two \emph{natural} theorems $A$ and $B$ that are predicatively reducible, while $A\wedge B$ is not.  
As a bonus, the proof of Theorem \ref{super2} also goes through using only intuitionistic logic.  
While $\LIN_{4}$ is quite natural (and implied by Lindel\"of's 1903 lemma), the same is not immediately clear for $(\exists \Xi)\LIN(\Xi)$, though a case can be made: $\Xi$ is essentially a realiser for paracompactness (as shown in \cite{sahotop}), and the latter seems to be essential 
for proving metrisation theorems, as suggested by the results in \cite{mummymf}*{Lemma 4.10}.
\end{rem}

\begin{ack}\rm
Our research was supported by the John Templeton Foundation, the Alexander von Humboldt Foundation, LMU Munich (via the Excellence Initiative and the Center for Advanced Studies of LMU), and the University of Oslo.
We express our gratitude towards these institutions. 
We thank Ulrich Kohlenbach, Karel Hrbacek, and Anil Nerode for their valuable advice.  We also thank the anonymous referee for the helpful suggestions.  
Opinions expressed in this paper do not necessarily reflect those of the John Templeton Foundation.    
\end{ack}

\bye